\documentclass{amsart}
\usepackage[left=2cm, right=2cm]{geometry}
\usepackage[utf8]{inputenc}
\usepackage{amsfonts, amsthm, amssymb, amsmath, mathtools,etoolbox}
\usepackage{enumitem}
\usepackage{blindtext}
\usepackage[colorlinks=true, urlcolor=blue, linkcolor=blue, citecolor=blue]{hyperref}
\usepackage{tikz,tikz-cd}
\usepackage{cleveref}
\usepackage{xcolor}
\usepackage{array}
\usepackage{graphicx}
\usepackage{extpfeil}
\tikzset{pullback/.style={minimum size=1.2ex,path picture={
\draw[opacity=1,black,-,#1] (-0.5ex,-0.5ex) -- (0.5ex,-0.5ex) -- (0.5ex,0.5ex);%
}}}
\usetikzlibrary{calc}

\theoremstyle{plain}
\newtheorem{theorem}{Theorem}[subsection]
\newtheorem{proposition}[theorem]{Proposition}
\newtheorem{lemma}[theorem]{Lemma}

\newtheorem{claim}{Claim}[theorem]
\crefname{claim}{Claim}{Claims}

\theoremstyle{definition}
\newtheorem{example}[theorem]{Example}
\newtheorem{definition}[theorem]{Definition}
\newtheorem{remark}[theorem]{Remark}
\newtheorem{notation}[theorem]{Notation}
\newtheorem{puzzle}[theorem]{Puzzle}

\newcommand{\dq}[1]{``#1"}

\newcommand{\invmemo}[1]{}
\newcommand{\Z}{\mathbb{Z}}

\newcommand{\R}{\mathbb{R}}

\newcommand{\E}{\mathcal{E}}
\newcommand{\F}{\mathbf{F}}
\renewcommand{\L}{\mathcal{L}}
\newcommand{\id}{\mathrm{id}}
\newcommand{\op}{\mathrm{op}}

\newcommand{\Set}{\mathbf{Set}}
\newcommand{\sSet}{\mathbf{sSet}}

\newcommand{\PSh}{\mathbf{PSh}}

\newcommand{\demph}[1]{\textit{#1}}

\newcommand{\Pow}{\mathcal{P}}
\newcommand{\Graph}{\mathrm{Graph}}
\newcommand{\Eq}{\mathrm{Eq}}

\newcommand{\sk}{\mathrm{sk}}
\newcommand{\cosk}{\mathrm{cosk}}
\newcommand{\uP}{\overline{P}}
\newcommand{\dP}{\underline{P}}
\newcommand{\mass}{\mathrm{mass}}
\newcommand{\uG}{\overline{G}}

\newcommand{\CrefIneq}{Inequalities \ref{eq:all_ineq}}
\newcommand{\Image}{\mathrm{Im}}
\title{Lawvere's fourth open problem:\\ Levels in the topos of symmetric simplicial sets}
\author{Ryuya Hora}
\address{Graduate School of Mathematical Sciences, University of Tokyo, Tokyo, Japan}
\email{hora@ms.u-tokyo}
\author{Yuhi Kamio}
\address{Graduate School of Mathematical Sciences, University of Tokyo, Tokyo, Japan}
\email{emirp13@g.ecc.u-tokyo.ac.jp.}
\author{Yuki Maehara}
\address{Research Institute for Mathematical Sciences, Kyoto University, Kyoto, Japan}
\email{ymaehar@kurims.kyoto-u.ac.jp}
\date{\today}

\begin{document}
\begin{abstract}
    %We prove that the Aufhebung relation, in the sense of William Lawvere, of the topos of symmetric simplicial sets is $2n-1$ for $n\geq 3$.
    %We compute the Aufhebung relation, in the sense of Lawvere, of the topos of symmetric simplicial sets.
    In the topos of simplicial sets, it makes sense to ask the following question about a given natural number $n$: \emph{what is the minimum value $m$ such that $n$-skeletality implies $m$-coskeletality?}
    This is an instance of the \emph{Aufhebung} relation in the sense of Lawvere, who introduced this notion for an arbitrary Grothendieck topos $\E$ in place of $\sSet$, and levels/essential subtopoi in place of dimensions.

    We compute this Aufhebung relation for the topos of symmetric simplicial sets.
    In particular, we show that it is given by $2l-1$ for the level labelled by $l\geq 3$, which coincides with the previously known case of simplicial sets.
    This result provides a solution to the fourth of the seven open problems in topos theory posed by Lawvere in 2009.
\end{abstract}
\maketitle

\tableofcontents

\section{Introduction}

\subsection{Levels and Aufhebung in a topos}
% \memo{\cite{lawvere2006some}}
As outlined in \cite{lawvere2006some},
the concept of dimension in a \dq{topos of spaces} may be captured by a special class of its subtopoi called \demph{levels}. A level of a Grothendieck topos $\E$ is an adjoint triple $l_! \dashv l^* \dashv l_* \colon \L \to \E$ with fully faithful $l_!$ and $l_*$. In other words, a level of $\E$ is an essential subtopos of $\E$.
% whose inverse image functor $\iota^*$ further admits a left adjoint. 
In \cite{kelly1989complete}, Kelly and Lawvere showed that the levels of a fixed Grothendieck topos form a complete lattice under the natural ordering of subtopoi. 
In the case of simplicial sets $\E=\sSet$, the notion of levels coincides with the usual notion of dimensions $-\infty < 0<1< \dots < \infty$. 
The familiar notion of $n$-skeletality can be extended to an arbitrary level $l_! \dashv l^* \dashv l_* \colon \L \to \E$, where we call an object $X$ \demph{$l$-skeletal} if the counit $l_!  l^* X \to X$ is an isomorphism. Similarly, an object $X$ is \demph{$l$-coskeletal} if the unit $X \to l^* l_* X$ is an isomorphism.

% [Here, a commutative diagram of adjoint functors is inserted using TikZ-CD.]

%What we consider in this paper is the \demph{way above} relation between levels.
This paper is primarily concerned with the \demph{way above} relation between levels.
Given two levels $l_1\geq l_0$ of a topos $\E$, we say that $l_1$ is way above $l_0$ if every $l_0$-skeletal object is $l_1$-coskeletal.
% \inred{[Do you want to swap $m$ and $l$ so that it becomes consistent with the next sentence?]} 
The \demph{Aufhebung} of a level $l$, if it exists, is the minimum level way above $l$. The Aufhebung of levels has been computed for some topoi, including graphic topoi \cite{lawvere1989display,lawvere1991more}, the topos of ball complexes \cite{roy1997topos}, and the topoi of simplicial and cubical sets \cite{zaks1986does, kennett2011levels}.

The purpose of this paper is to compute the Aufhebung for the presheaf topos $\PSh(\F)$, where $\F$ is the category of non-empty finite sets, which has long been an open problem 
% since \cite{lawvere1988toposesGenerated} 
in topos theory
as we now review.
% \memo{Aufhebung, and Way above relation. Solved for graphic topoi, Ball complexes, Simplicial and cubical sets,}

\subsection{The open problem and its history}
% \inred{I feel like I want to move the first sentence to the next paragraph.}\horamemo{I did it.}

% This topos 
%The presheaf topos $\PSh(\F)$, where $\F$ is the category of non-empty finite sets,
The topos $\PSh(\F)$ has attracted interest under two contrasting names that reflect its logical and geometric aspects respectively: \demph{the Boolean algebra classifier/the classifier of non-trivial Boolean algebras} \cite{lawvere1988toposesGenerated, roy1997topos, lawvere2025open, menni2019monic, menni2024successive}, a name derived from categorical logic, and the category of \demph{symmetric simplicial sets/ensembles simpliciaux sym\'{e}triques} \cite{grandis2001finite, rosicky2003left, cisinski2006prefaisceaux, roberts2009theory, hackney2025partial}, a term rooted in combinatorial topology.
% \footnote{
%Of course, these two aspects are interconnected;
Being two aspects of the same topos, of course these notions are interconnected;
their relationship is described in \cite{lawvere1988toposesGenerated} through a non-trivial Boolean algebra structure on the infinite-dimensional sphere.
% in the paper \cite{lawvere1988toposesGenerated}, Lawvere described their relationship through a non-trivial Boolean algebra structure on the infinite-dimensional sphere.
% }. 

% Two decades later, 

The Aufhebung of this topos $\PSh(\F)$ has been of interest at least since 1988 \cite{lawvere1988toposesGenerated, roy1997topos, menni2019monic, menni2024successive} from the viewpoint of dimension theory.
Its facet as a problem concerning \dq{complexity of automata} was pointed out in \cite{lawvere2004functorial}.
% \memo{Aufhebung of this particular topos is mentioned in 
% \cite{lawvere1988toposesGenerated, roy1997topos, lawvere2004functorial, menni2019monic, menni2024successive}.
% }
%\subsection{Combinatorial interpretation of the Aufhebung}
%As suggested in \cite{lawvere2004functorial}, the Aufhebung of the topos $\PSh(\F)$ can be regarded as a problem of combinatorial complexities.
%In our paper, we will regard symmetric sets (not only as simplicial sets with symmetric group actions, but also) 
%Regarding a symmetric set $M$ as a variant of Joyal’s \demph{combinatorial species} \cite{joyal1981theorie}, its skeletality and coskeletality may be interpreted as distinct quantitative notions of complexity.
% for combinatorial objects.
This paper adopts the latter viewpoint, but instead of automata, we will think of it as a problem concerning complexity of a variant of Joyal’s \demph{combinatorial species} \cite{joyal1981theorie}.
Accordingly, the notions of skeletality and coskeletality are interpreted as follows.
\begin{itemize}
    \item Skeletality of $M$ measures the extent to which an arbitrary $M$-structure may be recovered from small \demph{quotient structures}. (See \Cref{fig:EZ-decomposition}.) 
    % \memo{Write an example. For example, symmetric set of subset is $(2-1)$-skeletal, since every subset is a pullback of $\{\top\} \rightarrowtail\{\top,\bot\}$.}
    \item Coskeletality of $M$ measures the extent to which an arbitrary $M$-structure may be recovered from small \demph{substructures}.  (See \Cref{fig:CycleFilling}.) 
    % \memo{Write an example. For example, symmetric set of preorders is $(3-1)$-coskeletal, since every preorder is glued together from $3$-element preorder, which witnesses the transitivity axiom.}
\end{itemize}
Thus, the Aufhebung is the quantitative interplay between these dual complexity measures.

In 2009, Lawvere re-emphasised the significance of the Aufhebung of this topos by including it as Problem 4 in his list of seven open problems in topos theory:
\begin{quote}
    \cite{lawvere2025open}
[...] There is another fundamental topos related to classical constructions
and combinatorial topology, namely the Boolean algebra classifier that consists of presheaves on the category of finite non-empty sets. 
[...]
% Here again, the levels correspond to natural numbers (together with minus infinity), so that there are sufficient chain conditions to ensure that the jump operator exists. 
What is, in combinatorial or number-theoretic terms, the way below relation for this basic topos?
\end{quote}
% % This problem asks the Aufhebung 
% This topos has been studied \memo{hora: This topos was not studied in these papers. Some of them just mentioned the topos.} under various names, including \textit{the Boolean algebra classifier} in \cite{lawvere2025open, roy1997topos, lawvere1988toposesGenerated}, \textit{the classifier of non-trivial Boolean algebras} in \cite{menni2019monic, menni2024successive}, and the category of \textit{symmetric simplicial sets} in \cite{grandis2001finite, rosicky2003left, roberts2009theory, hackney2023partial}. 
%While the Aufhebung of this topos has been repeatedly mentioned in its long history,
% \cite{lawvere1988toposesGenerated, roy1997topos, lawvere2004functorial, lawvere2025open, menni2019monic, menni2024successive}
% Despite the amount of attention it has received, ha kowai
% the calculation of its Aufhebung
%it still remains unsolved.
For more explanations on related topics, see the early article \cite{lawvere1988toposesGenerated}, the recent paper \cite{menni2024successive}, and the papers cited therein.
% This problem was also discussed in \cite{roy1997topos, lawvere2004functorial}. 

\subsection{Contribution}
The present paper proves that, in the topos $\PSh(\F)$,
% , proving that the Aufhebung of $n$ is $2n-1$ for $n\geq 3$. 
%Our conclusion is
the Aufhebung of $l$ is given by 
\[
\begin{cases}
    0 &(l= -\infty )\\
    1 &(l= 0)\\
    2 &(l= 1 )\\
    4 &(l= 2)\\
    2l-1 &(l\geq 3)\\
    \infty & (l= \infty),
\end{cases}
\]
which coincides with the case of simplicial sets proven in \cite{kennett2011levels} for $l\geq 3$. 
%Our methodology is also inspired by theirs, relying on the Pigeonhole Principle and using degeneracy.
In fact, our proof, and in particular the way we utilise the Pigeonhole Principle, is inspired by this simplicial counterpart too. 
%However, our approach reduces the problem to a graph-theoretic phenomenon (\Cref{prop:GraphCalculation}), specifically to what we call a \demph{propagative graph}, which may provide new insight on why Aufhebung doubles the dimension.

A notable difference between our approach and that in \cite{kennett2011levels} is that we frame some of the combinatorics as a purely graph-theoretic phenomenon (\Cref{prop:GraphCalculation}), which provides a conceptual insight into why the Aufhebung doubles the dimension.
% Both strategy and result are similar to those of \cite{kennett2011levels}, which shows that the Aufhebung of simplicial sets is $2n-1$. What differs is that 

In fact, this use of graph theory is our response to the following phenomenon, which forced us to pay close attention to relationships between points in the underlying set $A$ of a given $M$-structure $x \in M(A)$: when proving coskeletality of the (suitably skeletal) symmetric simplicial set $M$, we need \emph{two} substructures to recover the given $M$-structure (see the construction of $f$ in \Cref{lem:fillingEdge}) whereas only \emph{one} was needed in the simplicial case (see \cite[Proposition 3.14]{kennett2011levels} where the filler is simply given by $c_m\sigma^m$).
The reason for this contrast can be traced back to the fact that, given a morphism $\alpha \colon [k] \to [n]$ in the simplex category $\Delta$ and $m \in [k]$, we have
\[
m = \min\Bigl\{ i \in [k] \mid \alpha^{-1}\bigl(\alpha(i)\bigr) \neq \{i\}\Bigr\}
\implies
\alpha(m) = \alpha(m+1)
\]
while we have no such control over functions in $\F$.

% \memo{\cite{menni2019monic} calculates $a_0$.}
% \memo{What to cite:  \cite{lawvere2025open, lawvere2006some}}
% \cite{menni2024successive}

\subsection*{Acknowledgement}
% We would like to thank Matias Menni for his encouraging advice.
We would like to thank Matias Menni for his encouraging advice on this problem.
We are also grateful to the Kanda Satellite Lab at the National Institute of Informatics for providing a space for research discussions.

The first-named author would like to thank his supervisor Ryu Hasegawa for helpful discussions and suggestions. 
% He also would like to appreciate Matias Menni's helpful and encouraging advice. 
He was supported by JSPS KAKENHI Grant Number JP24KJ0837 and FoPM, WINGS Program, the University of Tokyo.

The third-named author gratefully acknowledges the support of JSPS Research Fellowship for Young Scientists and JSPS KAKENHI Grant Number JP23KJ1365.

\section{Symmetric sets}
The main result of this paper concerns the \demph{Aufhebung} of the topos of \demph{symmetric sets}.
The primary purpose of this section is to recall these two notions.
We also construct a family of rather simple symmetric sets which provides a (sharp) lower bound for the Aufhebung.

\subsection{Definition and examples of symmetric sets}

\begin{definition}
    We will write $\F$ for the category of non-empty finite sets and all functions.
    The category of \demph{symmetric (simplicial) sets} and \demph{symmetric (simplicial) maps} is its presheaf category $\PSh(\F) = [\F^\op, \Set]$.
\end{definition}
% The category of symmetric sets is equivalent to the category of symmetric sets.

% \footnote{The notion of combinatorial species is originally introduced in \cite{joyal1981theorie}, as a functor from the core groupoid of $\FinSet$ to the category $\FinSet$.}
One can think of symmetric sets either
\begin{itemize}
    \item as a variant of simplicial sets equipped with an action of the symmetric groups, or
    \item as a variant of Joyal's combinatorial species \cite{joyal1981theorie}.
\end{itemize}
Although we will employ some key concepts motivated by the former viewpoint (such as the EZ-decomposition: \Cref{def:EZ-decomposition}), our arguments will be framed according to the latter.
In particular, given a symmetric set $M$ and a non-empty finite set $A$, we will refer to elements of $M(A)$ as \emph{$M$-structures} on $A$.
Each $M$-structure $x \in M(B)$ can be pulled back along an arbitrary function $\alpha \colon A \to B$, and we denote the resulting $M$-structure by $x\alpha \in M(A)$.

% Our combinatorial intuition for a symmetric set is as follows.
% \begin{itemize}
%     \item A symmetric set $M$ assigns, to each non-empty finite set $A$, the set $M(A)$ of all \inred{possible?} ``$M$-structures" on the underlying set $A$.
%     \item An $M$-structure $x \in M(B)$ on a set $B$ can be pulled back along a function $\alpha\colon A \to B$, defining $x \alpha \in M(A)$.
% \end{itemize}

 % include
Below are some typical examples of symmetric sets.
% \begin{itemize}
    % \item The symmetric set of \demph{graphs} is the contravariant functor $\Graph\colon \F^\op \to \Set$ that sends a non-empty finite set $A$ to the set of all undirected graphs whose vertex set is $A$:
    % \[
    % A \mapsto \Graph(A) \coloneqq \Pow(\Pow_{1,2}(A)).
    % \] 
    % Here, $\Pow_{1,2}$ is the covariant functor sending a set to the set of all subsets with cardinality $1$ or $2$, and $\Pow$ is the contravariant powerset functor. \inred{I don't understand this description. What are the singletons doing?}
    % % The notion of stabilizer is defined in the obvious way .
    % \item 
    \begin{example}[Graphs]\label{exmpl:SymmetricSetOfGraphs}
    % The symmetric set of \demph{graphs} is the contravariant functor $\Graph\colon \F^\op \to \Set$ that sends a non-empty finite set $A$ to the set of all undirected graphs whose vertex set is $A$.
    % Here, an undirected graph is formally defined as a pair $\bigl(A, E\subset \Pow_2(A)\bigr)$, where $\Pow_2(A)$ is the set of all $2$-element subsets of $A$.
    The symmetric set of \demph{graphs} is the contravariant functor $\Graph\colon \F^\op \to \Set$ that sends a non-empty finite set $A$ to the set of all undirected graphs whose vertex set is $A$, i.e., the set of all pairs $\bigl(A, E\subset \Pow_2(A)\bigr)$ where $\Pow_2(A)$ is the set of all $2$-element subsets of $A$.
    For a graph $x=\bigl(B, E \subset \Pow_2(B)\bigr) \in \Graph(B)$ and a function $\alpha\colon A \to B$, the pullback $x\alpha \in \Graph(A)$ is defined to be the graph $\bigl(A, \{e\in \Pow_2(A) \mid \alpha(e)\in E\}\bigr)$.
    See \Cref{fig:EZ-decomposition} for instance.
    (In what follows, $\Graph$ will be our default example when visualising concepts around symmetric sets.)
    % \item 
    \end{example}

    \begin{example}[Equivalence relations]\label{exmp:Equiv}
    The symmetric set of \demph{equivalence relations} $\Eq$ is defined by
    \[
    \Eq(A) \coloneqq \{{\sim}\subset A^2\mid {\sim}\text{ is an equivalence relation}\}.
    \]
    A variant of this example will be used to give a lower bound for the Aufhebung (\Cref{lem:LowerBoundcycle}).
    % Its $(n-1)$-skelton is $\Eq_n$ given by
    % \[
    % \Eq_n(A) \coloneqq \{{\sim}\subset A^2\mid {\sim}\text{ is an equivalence relation, and } |A/{\sim}|\leq n\}.
    % \]
% \end{itemize}
    \end{example}

    \begin{example}[Representable symmetric set $=$ colouring]
        For a non-negative integer $k$, we define the symmetric set $\Delta^k$ to be the presheaf $\F(-, S)$ represented by an $(n+1)$-element set $S$. Geometrically, one can think of $\Delta^k$ as a standard $k$-simplex. Combinatorially, a $\Delta^k$-structure on a non-empty finite set $A$ is simply a function $A \to S$, which may be thought of as an $S$-colouring of $A$.
    \end{example}

\subsection{EZ-decomposition and EZ-congruence}

% \memo{Write something about \cite[Definition 6.6, Proposition 6.7]{berger2011extension}}

Either by direct combinatorics or using \cite[Proposition 4.2 and Theorem 5.6]{campion2023cubical}, one can verify that the category $\F$ is an \demph{EZ-category}
% , Eilenberg-Zilber category,
in the sense of \cite[Definition 6.7]{berger2011extension} (which allows non-trivial automorphisms, unlike e.g.~\cite[Definition 1.3.1]{cisinski2019higher}). %Applying \cite[Proposition 6.9]{berger2011extension} to $\F$, we can observe that the following decomposition is essentially unique.
Thus, we are naturally led to the following notions.
% \horamemo{Menni's decomposition might be cited here.}

\begin{definition}[EZ-decomposition] \label{def:EZ-decomposition}
    Let $M$ be a symmetric set, $A$ be a non-empty finite set, and $x\in M(A)$ be an $M$-structure.
    \begin{itemize}
        \item A \demph{decomposition} of $x$ is a pair $(\alpha, y)$ consisting of a morphism $\alpha \colon A \to B$ in $\F$ and $y\in M(B)$ such that $x=y\alpha$. 
        \item We say $x$ is \demph{degenerate} if it admits a decomposition $(\alpha,y)$ where $\alpha$ is a non-invertible surjection.
        Otherwise $x$ is said to be \emph{non-degenerate}.
        \item A decomposition $(\alpha, y)$ of $x$ is called an \demph{EZ-decomposition} if $\alpha$ is a surjection and $y$ is non-degenerate.
    \end{itemize}
\end{definition}

% \begin{definition}[EZ-decomposition] \label{def:EZ-decomposition}
%     Let $M$ be a symmetric set, $A$ be a non-empty finite set, and $x\in M(A)$ be an $M$-structure.
%     \begin{itemize}
%         \item A \demph{decomposition} of $x$ is a pair $(\alpha, y)$ consisting of a morphism $\alpha \colon A \to B$ in $\F$ and $y\in M(B)$ such that $x=y\alpha$. 
%         % \inred{Or should a decomposition be the expression $x=y\alpha$ itself?}
%         \item The \demph{mass} of $x$, denoted by $\mass(x)$, is the minimum value of $|B|$ among all decompositions $(\alpha\colon A \to B, y)$ of $x$.
%         \item A decomposition $(\alpha\colon A \to B, y)$ is called an \demph{EZ-decomposition} if $|B|=\mass(x)$. 
%         % is minimum among all decompositions.
%     \end{itemize}
%     We say $x$ is \demph{degenerate} if $\mass(x) < |A|$.
% \end{definition}

% \begin{remark}
%     Notice that, if $x= y\alpha$ is an EZ-decomposition of $x$, then $y$ is necessarily non-degenerate and $\alpha$ is necessarily surjective.
%     Conversely, any decomposition $x=y\alpha$ with non-degenerate $y$ and surjective $\alpha$ is necessarily an EZ-decomposition.
% \end{remark}

\begin{example}
    An EZ-decomposition of a graph (regarded as an $M$-structure on a seven-element set for $M=\Graph$ of \Cref{exmpl:SymmetricSetOfGraphs}) is visualised in \Cref{fig:EZ-decomposition}. (The colours indicate the associated \demph{EZ-congruence}, which will be defined in \Cref{def:EZcongruence} together with the term \emph{mass}.)
% \Cref{fig:EZ-decomposition} visualise an example of EZ-decomposition of a graph $\Graph(\{1, \dots, 7\})$.
\end{example}

\begin{figure}[ht]
    \centering
\begin{tikzpicture}
    % 半径設定
    \def\r{2} % 各グラフの描画範囲の半径
    \def\RcircleLeft{2.8} % 左のグラフを囲む円の半径
    \def\RcircleRight{2} % 右のグラフを囲む円の半径
    \def\Xshift{10} % 左右のグラフの間隔

    % 頂点の色を設定
    \def\vertexcolor#1{%
        \ifnum#1=0 red\else%
        \ifnum#1=5 red\else%
        \ifnum#1=1 blue\else%
        \ifnum#1=2 blue\else%
        \ifnum#1=3 green\else%
        \ifnum#1=6 green\else%
        \ifnum#1=4 orange\else%
        black\fi\fi\fi\fi\fi\fi\fi%
    }

    % 左側のグラフ
    \begin{scope}
        % 左のグラフを囲む円
        \draw[thick] (0, 0) circle (\RcircleLeft);
        \node[below] at (0,-\RcircleLeft) {$x= y \alpha$};

        % 頂点を配置
        \foreach \i in {0,...,6} {
            \node[draw, circle, fill=\vertexcolor{\i}, minimum size=7pt, inner sep=0pt] (v\i) at ({90 + \i * 360 / 7}:\r) {};
        }

        % 辺を描画
        \draw (v0) -- (v1);
        \draw (v0) -- (v2);
        \draw (v0) -- (v3);
        \draw (v0) -- (v4);
        \draw (v0) -- (v6);
        \draw (v1) -- (v5);
        \draw (v2) -- (v5);
        \draw (v3) -- (v4);
        \draw (v3) -- (v5);
        \draw (v4) -- (v5);
        \draw (v4) -- (v6);
        \draw (v5) -- (v6);
    \end{scope}

    % 右側のグラフ
    \begin{scope}[shift={(\Xshift,0)}] % 右に移動
        % 右のグラフを囲む円
        \draw[thick] (0, 0) circle (\RcircleRight);
        \node[below] at (0,-\RcircleRight) {$y$};

        % 頂点を正方形に配置
        \node[draw, circle, fill=red, minimum size=7pt, inner sep=0pt] (w0) at (-1, 1) {};
        \node[draw, circle, fill=blue, minimum size=7pt, inner sep=0pt] (w1) at (-1, -1) {};
        \node[draw, circle, fill=orange, minimum size=7pt, inner sep=0pt] (w2) at (1, -1) {};
        \node[draw, circle, fill=green, minimum size=7pt, inner sep=0pt] (w3) at (1, 1) {};

        % 辺を描画
        \draw (w0) -- (w1);
        \draw (w0) -- (w2);
        \draw (w0) -- (w3);
        \draw (w2) -- (w3);
    \end{scope}

    % 矢印を追加（左から右へ）
    \draw[->>, thick] (\RcircleLeft + 0.5, 0) -- (\Xshift - \RcircleRight - 0.5, 0) node[midway, above] {$\alpha$};

\end{tikzpicture}
    \caption{EZ-congruence and EZ-decomposition: A $7$-vertex graph $x$ with mass $4$.}
    \label{fig:EZ-decomposition}
\end{figure}
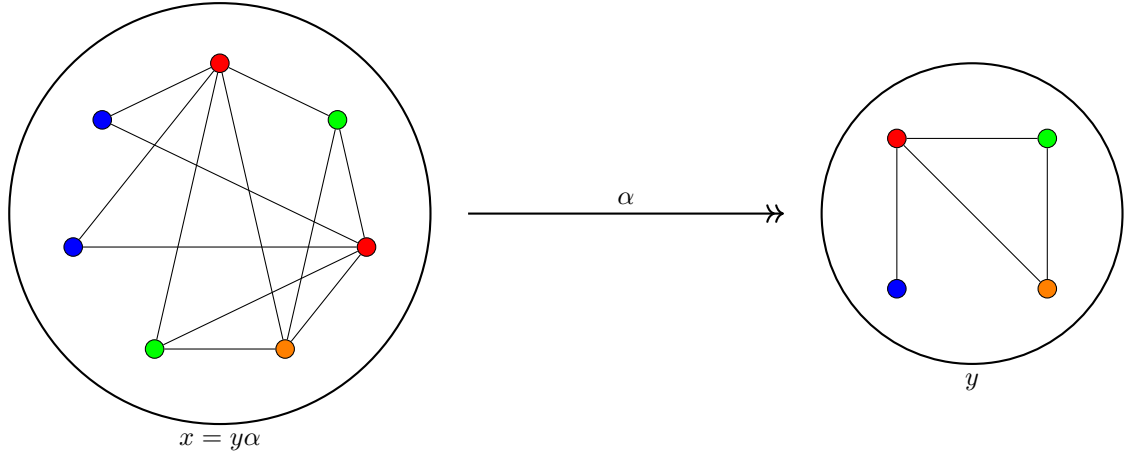

Each $M$-structure admits an essentially unique EZ-decomposition in the following sense.

\begin{proposition}[{Uniqueness of EZ-decomposition \cite[Proposition 6.9.]{berger2011extension}}] \label{prop:UniquenessOfEZdecomposition}
    Let $M$ be a symmetric set, $A$ be a non-empty finite set, and $x\in M(A)$ be an $M$-structure.
    % there exists a surjection $\beta \colon A \twoheadrightarrow B$ and $z\in M(A)$ such that $x=z\beta$. Furthermore, such 
    Then $x$ admits an EZ-decomposition.
    Moreover, for any two EZ-decompositions $(\alpha\colon A \twoheadrightarrow B, y)$ and $(\alpha'\colon A \twoheadrightarrow B', y')$ of $x$, there exists a (necessarily unique) bijection $\sigma \colon B \to
    % \xrightarrow{\cong}
    B'$ such that $\sigma \alpha = \alpha'$ and $y= y' \sigma$.
    \[
    \begin{tikzcd}
    &B\ar[dd, "\exists ! \sigma", dashed]& M(B)\ni y\\
        A\ar[ru,"\alpha", twoheadrightarrow]\ar[rd,"\alpha'"', twoheadrightarrow]&&\\
        &B'&M(B')\ni y'\ar[uu, mapsto, "M(\sigma)"']
    \end{tikzcd}
    \]
\end{proposition}

The essential uniqueness of EZ-decomposition guarantees that the following notions are well defined.

\begin{definition}[Mass and EZ-congruence]\label{def:EZcongruence}
    Let $M$ be a symmetric set, $A$ be a non-empty finite set, and $x\in M(A)$ be an $M$-structure with EZ-decomposition $(\alpha\colon A \twoheadrightarrow B, y)$.
    \begin{itemize}
        \item The \demph{mass} of $x$ is the cardinality $|B|$.
        \item The \demph{EZ-congruence} associated to $x$, which will be denoted by $\sim_x$, is the equivalence relation on $A$ defined by 
    \[
    a\sim_x b \iff \alpha(a)=\alpha(b).
    \]
    \end{itemize}
\end{definition}

% \begin{definition}[EZ-congruence]\label{def:EZcongruence}
%     For a symmetric set $M$ and an $M$-structure $x\in M(A)$ on a non-empty finite set $A$, the \demph{EZ-congruence} associated to $x$, which will be denoted by $\sim_x$, is the equivalence relation on $A$ defined by 
%     \[
%     a\sim_x b \iff \alpha(a)=\alpha(b),
%     \]
%     where $x=y\alpha$ is an EZ-decomposition of $x$.
% \end{definition}

The EZ-decompositions can be characterised as follows.

% \begin{lemma}\label{lem:CriterionOfEZDecomposition}
%     Let $M$ be a symmetric set, $A$ be a non-empty finite set, and $x\in M(A)$ be an $M$-structure.
%     Then a decomposition $(\alpha \colon A \to B,y)$ of $x$ is an EZ-decomposition if and only if $y$ is non-degenerate and $\alpha$ is surjective.
% \end{lemma}
% \begin{proof}
%     Suppose that $x=y \alpha$ is an EZ-decompostion.
%     Then, if $y$ is degenerate, $y$ admits a further decomposition $(\beta \colon  B \to C, z)$ with $|C|<|B|$.
%     The resulting decomposition $x= z(\beta\alpha)$ implies $\mass(x)\leq |C|<|B| = \mass(x)$, which is a contradiction.
%     If $\beta$ is not surjective, we have $\mass(x) \leq |\mathrm{Im}(\beta)|< |B|=\mass(x)$, which is again a contradiction.
%     This proves that $y$ is non-degenerate and $\alpha$ is surjective.
%     % , which contradicts the minimality of $|B|$.

%     Conversely, suppose that $y$ is non-degenerate and $\alpha\colon A \to B$ is surjective.
%     If $(y, \alpha)$ is not an EZ-decompostion,
%     % $$\mass(x)< |B|$, 
%     we can take an EZ-decomposition $(y', \alpha' \colon B' \to A)$, thus we have $|B'|= \mass(x)< |B|$. Taking a section $\sigma \colon B \rightarrowtail A$ of the surjection $\alpha$, we have $y= y \alpha \sigma = x\sigma = y' (\alpha' \sigma)$, which contradicts the assumption that $y$ is non-degenerate.
% \end{proof}

\begin{lemma}\label{lem:CriterionOfEZDecomposition}
    Let $M$ be a symmetric set, $A$ be a non-empty finite set, $x\in M(A)$ be an $M$-structure, and $(\alpha \colon A \to B,y)$ be a decomposition of $x$.
    Then we have $|B| \ge \mass(x)$, and moreover $(\alpha,y)$ is an EZ-decomposition if and only if $|B| = \mass(x)$.
\end{lemma}
\begin{proof}
    Let $\alpha = \mu\epsilon$ be an epi-mono factorisation of $\alpha$, and let $y\mu = z\beta$ be an EZ-decomposition of $y\mu$.
    \[
    \begin{tikzcd}
        A
        \arrow [r, "\epsilon", twoheadrightarrow]
        \arrow [dr, "\alpha"'] &
        C
        \arrow [r, "\beta", twoheadrightarrow]
        \arrow [d, "\mu", rightarrowtail] &
        D &
        x &
        y\mu
        \arrow [l, mapsto] & 
        z
        \arrow [l, mapsto] \\
        & B & & &
        y
        \arrow [u, mapsto]
        \arrow [ul, mapsto]&
    \end{tikzcd}
    \]
    Then we have $x = y\alpha = y\mu\epsilon = z(\beta\epsilon)$ and the rightmost expression provides an EZ-decomposition of $x$.
    Since $\beta$ is a surjection and $\mu$ is an injection, we have
    \[
    |B| \ge |C| \ge |D| = \mass(x).
    \]
    Moreover, the two equalities hold if and only if both $\mu$ and $\beta$ are invertible, or equivalently, if and only if $\alpha$ is a surjection (so that we may take $\mu = \id$) and $y$ is non-degenerate.
\end{proof}

% The next lemma will be utilized to fill a cycle.
% \begin{lemma}[Degeneracy data in face]\label{lem:reconstructionFromCompleteRepresentatives}
%     For a symmetric set $M$ and an $M$-structure $x\in M(A)$ on a non-empty finite set $A$, let $\iota \colon S\rightarrowtail  A$ be a complete set of representatives of the EZ-congruence of $x$. For the unique surjection $\sigma \colon A \twoheadrightarrow S$ such that $a\sim_x \sigma(a)$, we have $x \iota \sigma = x$. Furthermore, $(x\iota) \sigma$ is an EZ-decomposition of $x$.
% \end{lemma}

% So our problem of filling cycle is reconstructing the EZ-congruence from those of lower-dimensional cells. 

% One of the essential difficulties is that the EZ-congruence might increase by restriction:
\begin{lemma}\label{lem:laxCongruence}
Let $M$ be a symmetric set, $\phi \colon A \to A'$ be a morphism in $\F$, $a,b \in A$, and $x$ be an $M$-structure on $A'$.
If $\phi(a) \sim_x \phi(b)$ then we have $a\sim_{x\phi} b$.
\end{lemma}
\begin{proof}
    Let $x=y\alpha$ be an EZ-decomposition of $x$.
    Assume $\phi(a) \sim_x \phi(b)$, which is equivalent to the equation $\alpha\phi(a) = \alpha\phi (b)$.
    We wish to prove that we have $\beta(a) = \beta(b)$ for some (and equivalently all) EZ-decomposition $x\phi = z\beta$ of $x\phi$.
    
    Let $\alpha \phi = \mu\epsilon$ be an epi-mono factorisation of $\alpha \phi$, and let $y\mu = z\gamma$ be an EZ-decomposition of $y\mu$.
    \[
    \begin{tikzcd}
        A
        \ar[r,"\phi"]
        \ar[d,"\epsilon", twoheadrightarrow]
        \ar[dd, "\beta"', bend right, twoheadrightarrow]&
        A'
        \ar[d, twoheadrightarrow, "\alpha"] &
        x\phi &
        x
        \arrow [l, mapsto]\\
        B
        \ar[r,"\mu", rightarrowtail]
        \ar[d,"\gamma", twoheadrightarrow]&
        B' &
        y\mu
        \arrow [u, mapsto] &
        y
        \arrow [u, mapsto]
        \arrow [l, mapsto] \\
        C& &
        z
        \arrow [u, mapsto] &
    \end{tikzcd}
    \]
    Then we have
    \[
    x\phi = y\alpha\phi = y\mu\epsilon = z(\gamma\epsilon)
    \]
    and the rightmost expression is an EZ-decomposition of $x\phi$.
    %due to \Cref{lem:CriterionOfEZDecomposition}.
    It thus suffices to prove $\gamma\epsilon(a) = \gamma\epsilon(b)$, which indeed holds because
    \[
    \mu\epsilon(a) = \alpha\phi(a) = \alpha\phi(b) = \mu\epsilon(b)
    \]
    and $\mu$ is monic.
    % Let $x=z\alpha$ be an EZ-decomposition, $\alpha f = m e$ be its image factorization, and $zm = z'\beta$ be an EZ-decomposition. Then, $f(a) \sim_x f(b)$ implies $\alpha f(a)= \alpha f (b)$, $e(a) = e(b)$, and $\beta e (a) = \beta e (b)$. Since $xf = z\alpha f = z me = z' (\beta e)$ is an EZ-decomposition, this proves $a\sim_{xf} b$.
\end{proof}
% In particular, we obtain $\mass(x) \geq \mass(x\phi)$.

% \begin{remark}
    Notice that the converse of the above lemma does not hold in general.
    % \inred{Example?} 
    % \horamemo{
    The lower half of \Cref{fig:DecompositionLifting} provides a counterexample.
    % }
% \end{remark}

\begin{lemma}\label{lem:pbdoesntIncreaseMass}
Let $M$ be a symmetric set, $\phi \colon A \to A'$ be a morphism in $\F$, and $x$ be an $M$-structure on $A'$.
Then we have $\mass(x) \geq \mass(x\phi)$.
\end{lemma}
\begin{proof}
    This is an immediate corollary of \Cref{lem:laxCongruence}.
\end{proof}

\subsection{Levels in the topos of symmetric sets}
% This section aims to classify all the levels in the topos $\PSh(\F)$.

For a non-negative integer $l\geq 0$, we write $\F_l$ for the full subcategory of $\F$ consisting of the sets of size less than or equal to $l+1$. Geometrically, $\F_l$ is the category of standard symmetric simplices of dimension less than or equal to $l$. 
% As is often the case in combinatorial topology, this paper frequently uses such a \dq{shift by one} notation, which arises from the fact that the standard $n$-simplex has $n+1$ vertices. In other words, an $(n+1)$-element $M$-structure is $n$-dimensional.
We also write $\F_{-\infty}$ for the empty category, and $\F_{\infty}$ for $\F$ itself.

The following proposition is mentioned in \cite{lawvere1988toposesGenerated}. See also \cite[Corollary 3.3, Example 6.3]{menni2024successive}.
\begin{proposition}[Classification of levels \cite{lawvere1988toposesGenerated}]\label{prop:Levels}
    The complete lattice of levels in the topos $\PSh(\F)$ is isomorphic to the total order 
    \[
    -\infty < 0< 1< 2< \dots<\infty,     
    \]
    where the level labelled by \footnote{We avoid simply calling it ``level $l$'' because the terms ``level $0$'' and ``level $1$'' have specific meanings defined by the Aufhebung relation.} $l$ is the adjoint triple
    \[
    \begin{tikzcd}
        \PSh(\F_{l}) 
        \ar[r, shift left =10pt, "\mathrm{Lan}_{\iota_l}"]
        \ar[r, shift right=10pt, "\mathrm{Ran}_{\iota_l}"']
        &\PSh(\F)\ar[l, "{\iota_l^*}"']
    \end{tikzcd}
    \]
    induced by the inclusion $\iota_l \colon \F_l \to \F$.
\end{proposition}
% Notice that the level $l$ is induced by $\F_{l+1}$ not by $\F_l$.
% As is often the case in combinatorial topology, this paper frequently uses such a \dq{shift by one} notation, which arises from the fact that the standard $n$-simplex has $n+1$ vertices. In other words, an $(n+1)$-element $M$-structure is $n$-dimensional.

% \inred{Shouldn't there be a level corresponding to $\F_0$?}\horamemo{$\F_0 = \F_{-\infty}$, because we don't consider the empty set.}

\begin{remark}[All subtopoi are levels]
As mentioned in \cite{lawvere1988toposesGenerated}, in this special topos $\PSh(\F)$, all subtopoi are in fact of the form of \Cref{prop:Levels} and thus levels. This fact can be seen as an instance of either \cite[Corollary 3.3]{menni2024successive} or \cite[Corollary 4]{marques2024criterion}.
\end{remark}

% \subsection{Cycle filling}

Let us now recall the notions of \demph{skeleton} and \demph{coskeleton}.
\begin{definition}[Skeleton and coskeleton]
    For a level of $\PSh(\F)$ labelled by $l \in \{-\infty, 0, 1,  \dots ,\infty\}$, the induced comonad on $\PSh(\F)$ is denoted by $\sk_l$, and the induced monad is denoted by $\cosk_l$.
    A symmetric set $M$ is said to be \demph{$l$-skeletal} if the counit $\sk_l M \to M$ is invertible, and is said to be \demph{$l$-coskeletal} if the unit $M \to \cosk_l M$ is invertible.
\end{definition}

\begin{proposition}[Skeletality]\label{prop:characterizationOfSkeletality}
For a symmetric set $M$, its $l$-skeleton $\sk_l M$ is the symmetric subset $\epsilon_{M}\colon \sk_l M \rightarrowtail M$ consisting of the $M$-structures whose mass is less than or equal to $l+1$.
In other words, for any non-empty finite set $A$, we have
\[
(\sk_l M) (A) = \{x\in M(A)\mid \mass(x) \leq l+1\}\subset M(A).
\]
    In particular, a symmetric set $M$ is $l$-skeletal if and only if $\mass(x)\leq l+1$ for any non-empty finite set $A$ and any $M$-structure $x\in M(A)$.
\end{proposition}
\begin{proof}
This is a special case of \cite[Corollary 6.10]{berger2011extension}.
See also \cite[Proposition 2.1, Example 2.2]{menni2019monic}.
% which includes the example of the particular topos $\PSh(\F)$.
\end{proof}

To provide a concrete description of $l$-coskeletality, we need the notion of \demph{cycle}, whose definition in turn requires the following notation.
For a finite set $A$ with $|A|>1$ and an element $a\in A$, we write $\delta^a$ for the canonical embedding
\[
\delta^a \colon A\setminus \{a\} \rightarrowtail A.
\]
We will not notationally distinguish $\delta^a$'s  with different codomains.
For example, we write $\delta^a\delta^b = \delta^b \delta^a$ to express the commutativity of the square
\[
\begin{tikzcd}
    A\setminus\{a,b\} \ar[r,rightarrowtail, "\delta^a"]\ar[d,rightarrowtail, "\delta^b"]&A\setminus\{b\}\ar[d,rightarrowtail, "\delta^b"]\\
    A\setminus\{a\}\ar[r,rightarrowtail, "\delta^a"]&A.
\end{tikzcd}
\]

\begin{definition}[Cycle]\label{def:cycle}
    For a positive integer $k> 0$, a \demph{$k$-cycle} in a symmetric set $M$ is a pair $(P, \{c_p\}_{p\in P})$ of a $(k+1)$-element set $P$ and a family of $M$-structures $c_p \in M(P \setminus \{p\})$ that satisfies \demph{the cycle equation}
    \[
    c_p \delta^q = c_q \delta^p \in M(P\setminus \{p,q\})
    \]
    for any distinct elements $p,q \in P$ (except for the case $k=1$, in which case we do not enforce the cycle equation because $P\setminus\{p,q\}$ would be empty).
    A \demph{filler} of a cycle $(P, \{c_p\}_{p\in P})$ is an $M$-structure $f\in M(P)$ that satisfies $f\delta^p = c_p$ for every $p\in P$.
    We say $(P, \{c_p\}_{p\in P})$ is \demph{unfillable} if it admits no fillers.
\end{definition}
See \Cref{fig:CycleFilling} for a visualisation of cycle filling in the case where $M= \Graph$.

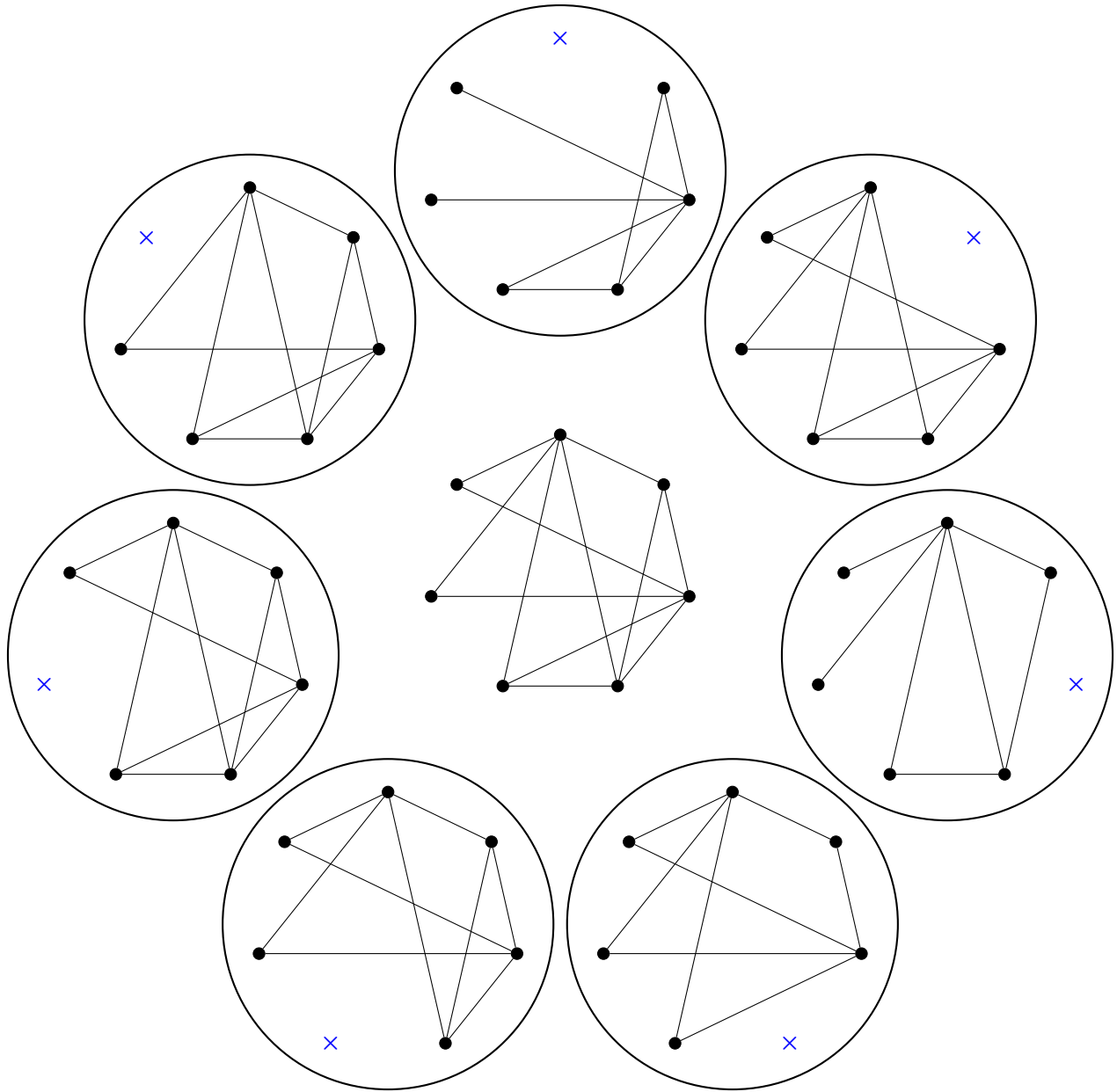
\begin{figure}[ht]
    \centering
    \begin{tikzpicture}
    % 半径設定
    \def\r{2} % 各グラフの描画範囲の半径
    \def\R{6} % 外周のコピーを配置する半径

    % 元のグラフ（中央）
    \begin{scope}
        % 中央のグラフを囲む円（色付き）
        % \draw[thick] (0, 0) circle (\r + 0.5);
        
        % 中央の頂点を描画
        \foreach \i in {0,...,6} {
            \node[draw, circle, minimum size=5pt, inner sep=0pt, fill] (v\i) at ({90 + \i * 360 / 7}:\r) {};
        }
        % 中央のグラフの辺を描画
        \draw (v0) -- (v1);
        \draw (v0) -- (v2);
        \draw (v0) -- (v3);
        \draw (v0) -- (v4);
        \draw (v0) -- (v6);
        \draw (v1) -- (v5);
        \draw (v2) -- (v5);
        \draw (v3) -- (v4);
        \draw (v3) -- (v5);
        \draw (v4) -- (v5);
        \draw (v4) -- (v6);
        \draw (v5) -- (v6);
    \end{scope}

    % 周囲にコピーを配置
    \foreach \j in {0,...,6} {
        % コピーを囲む円
        \begin{scope}[shift={({90 + \j * 360 / 7}:\R)}]
            \draw[thick] (0, 0) circle (\r + 0.5);
            
            % 各コピーの頂点を描画
            \foreach \i in {0,...,6} {
                \ifnum\i=\j
                    % 該当する頂点の場所に青い×マークを描画
                    \node[blue] at ({90 + \i * 360 / 7}:\r) {\textbf{\texttimes}};
                \else
                    \node[draw, circle, minimum size=5pt, inner sep=0pt, fill] (w\j\i) at ({90 + \i * 360 / 7}:\r) {};
                \fi
            }

            % 各コピー内の辺を描画
            \foreach \i/\k in {0/1, 0/2, 0/3, 0/4, 0/6, 1/5, 2/5, 3/4, 3/5, 4/5, 4/6, 5/6} {
                \ifnum\i=\j
                    % 該当する頂点は辺を描画しない
                \else
                    \ifnum\k=\j
                        % 該当する頂点は辺を描画しない
                    \else
                        \draw (w\j\i) -- (w\j\k);
                    \fi
                \fi
            }
        \end{scope}
    }
\end{tikzpicture}
    \caption{%A combinatorial example of cycle filling for $(n,k,d)=(3,6,3)$
    % $n=3, k=6, d=3$
    A filler of a $6$-cycle in $\Graph$
    }
    \label{fig:CycleFilling}
\end{figure}

% \begin{proposition}[Coskeletality]\label{prop:CharacterizationOfCoskeletality}
%     For a symmetric set $M$ and a level labelled by $l\in \{-\infty, 0, 1, \dots, \infty\}$, $M$ is $l$-coskeletal, if and only if, for any non-negative integer $k$ greater than $l$, every $k$-cycle in $M$ has a unique filler.
% \end{proposition}
% \begin{proof}
%     \memo{ prove. Check if it actually holds even in the lower dimensions.}
% \end{proof}

\begin{lemma}[Universality of \texorpdfstring{$\sk_l \Delta^k$}{skDelta}]\label{lem:UniversalityOfSkDelta}
    Let $k$ and $l$ be non-negative integers, and let $P$ be a $(k+1)$-element set representing $\Delta^k \cong \F(-,P)$.
    Then for each symmetric set $M$, there is a bijection between
    \begin{itemize}
        \item symmetric maps $f\colon \sk_l \Delta^k \to M$, and
        \item families of $M$-structures $\{x_S \in M(S)\}_{S\in \Pow_{\leq l+1}(P)}$ such that $x_S * (S \leftarrowtail S') = x_{S'}$ for all pairs $S' \subset S$ in $\Pow_{\leq l+1}(P)$ (where $\Pow_{\leq l+1}(P)$ is the set of non-empty subsets of $P$ of size at most $l+1$)
    \end{itemize}
    which is natural in $M$.
    Consequently, for a positive integer $k>0$, the symmetric maps $\sk_{k-1} \Delta^k \to M$ are in bijection with the $k$-cycles in $M$.
\end{lemma}
\begin{proof}
    Due to \Cref{prop:characterizationOfSkeletality}, $\sk_l(\Delta^k)(A)$ may be identified with the set of functions $ P \xleftarrow{\alpha} A$ with $|\Image(\alpha)|\leq l+1$.
    Given a symmetric map $f\colon \sk_l \Delta^k \to M$, we can construct a family $\{x_S \in M(S)\}_{S\in \Pow_{\leq l+1}(P)}$ by
    \[
    x_S \coloneqq f_S (P \leftarrowtail S) \in M(S).
    \]
    The naturality of $f$ implies that the resulting family satisfies the required compatibility condition.

    % This correspondence is injective, since $f_A (P \xleftarrow{\alpha} A)$ is reconstructed as $x_{\Image(\alpha)}* (\Image(\alpha) \twoheadleftarrow A)$. To prove the surjectivity, 
    
    Conversely, given an arbitrary family $\{x_S \in M(S)\}_{S\in \Pow_{\leq l+1}(P)}$ that satisfies the compatibility condition, we define a symmetric map $f\colon \sk_l \Delta^k \to M$ by
    \[
    f_A (P \xleftarrow{\alpha} A) \coloneqq x_{\Image(\alpha)}*(\Image(\alpha) \xtwoheadleftarrow{\alpha} A).
    \]
    % One can check that this in fact defines a morphism of symmetric sets.
    % \begin{align*}
    %     f_B(P \xleftarrow{\alpha} A \xleftarrow{\beta} B) &= x_{\Image{(\alpha\circ  \beta)}} * (\Image (\alpha \circ \beta) \xtwoheadleftarrow{\alpha \circ \beta} B)\\
    %     &= x_{\Image{(\alpha)}} *(\Image{(\alpha)}\leftarrowtail{\Image(\alpha \circ \beta)} )* (\Image (\alpha \circ \beta) \xtwoheadleftarrow{\alpha \circ \beta} B)\\
    %     &= x_{\Image{(\alpha)}} *(\Image{(\alpha)}\xtwoheadleftarrow{\alpha} A)* (A\xleftarrow{\beta} B)\\
    %     &= f_A(P \xleftarrow{\alpha} A) * ( A \xleftarrow{\beta} B).
    % \end{align*}
    % It is easy to observe that each construction is the inverse constructions to the other. The naturality is also easier to observe.
    It is straightforward to verify that this indeed yields a symmetric map, and that it is moreover inverse to the construction in the previous paragraph.
    The naturality is also clear.

    The last statement is a consequence of the observation that, in the case $l=k-1$, the compatible families $\{x_S \in M(S)\}_{S\in \Pow_{\leq k}(P)}$ are in bijection with the $k$-cycles via $c_p \coloneqq x_{P\setminus \{p\}}$.
\end{proof}

\begin{proposition}[Coskeletality]\label{prop:CharacterizationOfCoskeletality}
    Let $M$ be a symmetric set and $l$ be a non-negative integer.
    Then the following are equivalent.
    \begin{itemize}
        \item[(1)] $M$ is $l$-coskeletal.
        \item[(2)] For any positive integer $k$ greater than $l$, every $k$-cycle in $M$ has a unique filler.
    \end{itemize}
\end{proposition}
\begin{proof}
We consider two additional conditions $(1'), (2')$ as follows.
\begin{itemize}
    \item[$(1')$] $M$ is right orthogonal to the counit $\epsilon_{\Delta^n} \colon \sk_l \Delta^n \rightarrowtail \Delta^n$ for any non-negative integer $n\geq 0$.
    \item[$(2')$] $M$ is right orthogonal to the counit $\epsilon_{\Delta^k} \colon \sk_{k-1} \Delta^k \rightarrowtail \Delta^k$ for any $k> l$.
\end{itemize}
We have $(1) \iff (1')$ because
\begin{align*}
    (1) & \iff \text{the unit }\eta_M \colon M \to\cosk_l M \text{ is an isomorphism}\\
    & \iff \text{for any } n\geq 0, \; (\eta_M)_* \colon \PSh(\F)(\Delta^n, M) \to \PSh(\F)(\Delta^n, \cosk_l M) \text{ is a bijection}\\
    & \iff \text{for any } n\geq 0, \; (\epsilon_{\Delta^{n}})^* \colon \PSh(\F)(\Delta^n, M) \to \PSh(\F)(\sk_l \Delta^n, M) \text{ is a bijection}\\
    & \iff (1').
\end{align*}
% The proof of the equivalences $(1) \iff (1)'$ and $(2) \iff (2)'$ are straightforward. 
The equivalence $(2) \iff (2')$ follows from \Cref{lem:UniversalityOfSkDelta}.
% , it suffices to prove that $\sk_{k-1}\Delta^k$
%It thus suffices to prove $(1') \iff (2')$.
    % \memo{ prove. Check if it actually holds even in the lower dimensions.}
    % Assuing $(1')$, for any $k>n$ we can deduce that $M$ is right orthogonal to  $\epsilon_{\Delta^n} \colon \sk_{k-1} \Delta^k \rightarrowtail \Delta^k$ for any $k'\geq n$, since $n$-coskeletality implies $(k-1)$-coskeletality.
    % The implication $(1) \implies (2')$ follows from the equivalence $(1) \iff (1')$ and the fact that the $l$-coskeletality implies the $(k-1)$-coskeletality when $k-1\geq l$. \inred{Reference/comment?}\horamemo{This follows since $\F_{l}\subset\F_{k-1}$.}

    Now we prove $(1) \implies (2')$.
    Assume $(1)$ and fix $k > l$.
    Note that, since we have the inclusion $\F_{l}\subset\F_{k-1}$, the $l$-coskeletal symmetric set $M$ is also $(k-1)$-coskeletal.
    In particular, using an instance of the equivalence $(1) \iff (1')$, we can deduce that $M$ is right orthogonal to the counit $\epsilon_{\Delta^k} \colon \sk_{k-1} \Delta^k \rightarrowtail \Delta^k$ as desired.

    It remains to prove $(2')\implies (1')$.
    So assume $(2')$ and consider a symmetric map $\sk_l \Delta^n \to M$ with $n > l$ (note that $(1')$ is trivial in the case $n \le l$).
    %We can construct the lift along  $\epsilon_{\Delta^n} \colon \sk_n \Delta^k \rightarrowtail \Delta^k$.
    % by decomposing it into a chain of morphisms 
    % \[
    % % \sk_n \Delta^k \rightarrowtail \sk_{n+1} \Delta^k \rightarrowtail \dots \rightarrowtail \sk_k \Delta^k = \Delta^k. 
    % \]
    %Take an arbitrary morphism $\sk_n \Delta^k \to M$.
    By \Cref{lem:UniversalityOfSkDelta}, we may identify this symmetric map with a compatible family $\{x_{S}\}_{S\in \Pow_{\leq l+1} (P)}$ of $M$-structures where $P$ is some $(n+1)$-element set.
    %By applying the lifting property along $\epsilon_{\Delta^{n+1}} \colon \sk_{n} \Delta^{n+1} \rightarrowtail \Delta^{n+1}$ to every $n+2$ element subset of $P$, we obtain the unique extension of the family $\{x_{S}\}_{S\in \Pow_{\leq n+1} (P)}$ to a new compatible family $\{x_{S}\}_{S\in \Pow_{\leq n+2} (P)}$.
    The assumption $(2')$ in the case of $k=l+1$ implies that there is a unique extension of this compatible family to $\{x_{S}\}_{S\in \Pow_{\leq l+2} (P)}$ indexed by the non-empty subsets of $P$ of size at most $l+2$ (as opposed to $l+1$).
    By repeating this process, we eventually obtain a unique extension $\{x_{S}\}_{S\in \Pow_{\leq n+1} (P)}$, which corresponds to a symmetric map $\Delta^n = \sk_{n} \Delta^n \to M$.
\end{proof}

\begin{remark}[Combinatorial intuition of coskeletality]
    To clarify the intuition behind coskeletality, which may be somewhat difficult to grasp, let us give an informal explanation.
    % A symmetric set $M$ is $n$-coskeletal 
    % if an $M$-structure is uniquely glued together from their structures on $n+1$-element subsets.
    A symmetric set $M$ is $l$-coskeletal if, for
    \begin{itemize}
        \item any non-empty finite set $A$ and
        \item any compatible family of $M$-structures on all subsets of $A$ with at most $l+1$ elements,
    \end{itemize} 
    there exists a unique way to glue them together to an $M$-structure on the entire set $A$.
    This is an example of the sheaf
    % ification 
    condition. 
    % \inred{[sheaf condition?]}
    In other words, the $l$-coskeletality of $M$ means that (both the existence and the equality of) $M$-structures on $A$ are completely determined by their \demph{substructures} on subsets $B \subset A$ of size at most $l+1$, while the $l$-skeletality means that they are determined by their \demph{quotient structures}.

    For certain examples, it makes sense to think of the coskeletality as measuring how many variables are needed in order to describe what an $M$-structure is.
    For instance, the symmetric set $\Eq$ of equivalence relations is $2$-coskeletal (but not $1$-coskeletal), reflecting the fact that the transitivity condition $x\sim y \land y \sim z \implies x\sim z$ involves $(2+1) = 3$ variables.
    Another such example is the representable $\Delta^n$, which is the symmetric set of $\{0,1, \dots , n\}$-valued functions (without any additional structures nor conditions).
    Since a function is completely determined by its values at individual points, this symmetric set is $0$-coskeletal.
\end{remark}

\subsection{Lower bound for the Aufhebung}

\begin{definition}[Aufhebung]\label{def:Aufhebung}
    For a level of the topos $\PSh(\F)$ labelled by $l \in \{-\infty , 0,1, \dots , \infty\}$, its \demph{Aufhebung} is the least level $l\leq a_l\in \{-\infty , 0,1, \dots , \infty\}$ such that every $l$-skeletal symmetric set is $a_l$-coskeletal.
\end{definition}
% \memo{
% Usually, we also consider the condition that every $l$-coskeletal object is $a_l$-coskeletal, i.e., $l\leq a_l$. \memo{And this is one reason why it's called Aufhebung.} However, in the topos $\PSh(\F)$, those two definitions coincide.
% if we obtain the value of $a_l$ (according to \Cref{def:Aufhebung}), then we obtain the latter notion of Aufhebung by $\max(a_l, l)$, since the complete lattice of levels of $\PSh(\F)$ is totally ordered.
% }
% this condition becomes trivial (\memo{Write}).

% \memo{Write $a_{-\infty}$ and $a_0$.}

% For the reader's convenience, we give the calculations of the first two
% \inred{I didn't see how this would be convenient for the reader.}
% \horamemo{I'm not sure if this is convenient for the reader, but this is theoretically needed.}
Let us recall the first two
Aufhebung $a_{-\infty}$ and $a_{0}$, which are already known. See \cite[Section 2]{lawvere2006some} for more explanations and motivating ideas around the Aufhebung $a_{-\infty}$ and $a_{0}$.
\begin{proposition}\label{prop:AufhebungMinusInfty}
    % We have 
    $a_{-\infty}=0$.
\end{proposition}
\begin{proof}
    A symmetric set $M$ is $(-\infty)$-skeletal if and only if $M$ is the initial object, i.e., the empty symmetric set. Since the empty symmetric set is $0$-coskeletal but not $(-\infty)$-coskeletal, we have $a_{-\infty}=0$. 
\end{proof}

% \memo{
More generally, if $\E_{\lnot \lnot}$ is a level of $\E$, then it coincides with the Aufhebung $a_{-\infty}$ since it is the least dense subtopos
% $\E_{\lnot \lnot}$
% is the sheaf topos over the double negation topology 
(see \cite[A.4.5]{johnstone2002sketchesv1}).

\invmemo{cite: level 0 = double negation topology. In Lawvere's terminology, $a_{-\infty}$ is \dq{becoming}, which is the Aufhebung of \dq{nothing $\emptyset$} and \dq{pure being $1$}.
\href{https://ncatlab.org/nlab/show/dense+subtopos\#relation_to_aufhebung}{[nLab: dense subtopos relation to aufhebung]}}

The value of the Aufhebung $a_0$ was first established in \cite{lawvere1988toposesGenerated}.
The paper \cite{menni2019monic} contains a detailed explanation on ``one-dimensionality'' and also on the particular example of the topos $\PSh(\F)$.

\invmemo{cite Menni: monic skeleta paper\cite{menni2019monic} and Lawvere's \cite{lawvere2006some}}

\begin{proposition}[\cite{lawvere1988toposesGenerated}, \cite{menni2019monic}]\label{prop:AufhebungZero}
     % We have 
     $a_{0}=1$.
\end{proposition}
\begin{proof}
    A symmetric set $M$ is $0$-skeletal if and only if $M$ is a small coproduct of copies of the terminal object, i.e., if and only if $M$ is discrete.
    We can easily see that the two-point discrete symmetric set is not $0$-coskeletal, which implies $a_0\geq 1$.
    It suffices to prove that any discrete $M$ is $1$-coskeletal.
    Since any symmetric map with a discrete codomain must be constant on each connected component, the required orthogonality follows from the fact that $\sk_1\Delta^n \to \Delta^n$ induces a bijection between their connected components.
\end{proof}

% \begin{example}[$a_{-\infty}=0$]
%     Let's calculate $a_{-\infty}$. 
% \end{example}

In the rest of the present subsection, we will give a lower bound for the Aufhebung.

\begin{definition}[$\Eq_{\le l}$ and $\Eq_{=l+1}$]
    % \inred{Need to define $\Eq_n$. Also, suggestion: $\Eq_{\le n}$ and $\Eq_{=n+1}$ instead of $\Eq_n$ and $\Eq_{n,n+1}$.}
    For $l \ge 1$, we define the symmetric set $\Eq_{\le l}$ to be the $(l-1)$-skeleton of $\Eq$ (See \Cref{exmp:Equiv})
    \[
    \Eq_{\le l} \coloneqq \sk_{l-1} \Eq,
    \]
    and the symmetric set
    $\Eq_{=l+1}$ by the pushout diagram
    \[
    \begin{tikzcd}
         \Eq_{\le l} \ar[r, rightarrowtail ]\ar[d, twoheadrightarrow, "!"]&\Eq_{\le l+1}\ar[d, twoheadrightarrow]\\
         1 \ar[r] &\Eq_{=l+1}.\ar[lu, very near start, phantom, "\ulcorner"]
    \end{tikzcd}
    \]
\end{definition}
More concretely,
$\Eq_{\leq l}(A)$ is the set of all partitions of $A$ into at most $l$ non-empty parts,
\[
\Eq_{\le l}(A) = \bigl\{{\sim}\in \Eq(A)\mid |A/{\sim}|\le l \bigr\}.
\]
Similarly, $\Eq_{=l+1}(A)$ is the set of all partitions of $A$ into exactly $l+1$ non-empty parts, with an \dq{error message} $\ast$:
\[
\Eq_{=l+1}(A) = \bigl\{{\sim}\in \Eq(A)\mid |A/{\sim}|= l+1 \bigr\} \sqcup \{\ast\}.
\]
If pulling back an equivalence relation results in fewer than $l+1$ parts then it instead returns an error message, and pulling back an error message also results in an error message.

It is easy to see that $\Eq_{=l+1}$ is $l$-skeletal, and the following unfillable cycle provides the desired lower bound.

% \memo{We need to think about how we deal with the cases where $n=-\infty, 0,1,2, $}

% \memo{Geometrically, this is a quotient of the $n$-sphere $S^n$ by the $S_{n+1}$-action. }

\begin{lemma}\label{lem:LowerBoundcycle}
    Let $l \ge 1$.
    Then $\Eq_{=l+1}$ has an unfillable $k$-cycle where
    \[
    k =
    \begin{cases}
        2 &(l=1)\\
        4 &(l=2)\\
        2l-1&(l\geq 3).
    \end{cases}
    \]
\end{lemma}
\begin{proof}
    First, we construct an unfillable $2$-cycle on an $3$-element set $P$. (This proves the cases $l=1$.) For each $p\in P$, we define $c_p$ as the discrete equivalence relation on $P \setminus \{p\}$. 
    Since there are exactly $2$ equivalence classes, we may regard $c_p$ as an element of $\Eq_{=2}\bigl(P \setminus \{p\}\bigr)$.
    These $\Eq_{=2}$-structures clearly satisfy the cycle equations. 
    This cycle is unfillable because the only equivalence relation on $P$ that pulls back to $c_p$ for all $p \in P$ is the discrete one, which has $3$ equivalence classes.
    %First, we construct an unfillable $(l+1)$-cycle on an $(l+2)$-element set $P$. (This proves the cases $l=1$.) For each $p\in P$, we define $c_p$ as the discrete equivalence relation on $P \setminus \{p\}$. 
    %Since there are exactly $l+1$ equivalence classes, we may regard $c_p$ as an element of $\Eq_{=l+1}\bigl(P \setminus \{p\}\bigr)$.
    %These $\Eq_{=l+1}$-structures satisfy the cycle equations. 
    %This cycle is unfillable because the only equivalence relation on $P$ that pulls back to $c_p$ for all $p \in P$ is the discrete one, which has $l+2$ equivalence classes.
    
    Next, we construct an unfillable $4$-cycle on the $5$-element set $P= \Z/5\Z$. (This proves the cases $l=2$.) 
    For each $p\in \Z/5\Z$, define an equivalence relation $c_p$ on $P \setminus \{p\}$ by
    \[
    ac_p b \iff (a+1=p=b-1) \lor (b+1=p=a-1) \lor (a=b).
    \]
    In other words, $c_p$ partitions $P \setminus \{p\}$ into a single two-element set $\{p-1,p+1\}$ and $3$ singletons.
    Since there are exactly $4$ equivalence classes, we may regard $c_p$ as an element of $\Eq_{=4}\bigl(P \setminus \{p\}\bigr)$.
    These $\Eq_{=4}$-structures satisfy the cycle equations because, for distinct $p,q \in \Z/5\Z$,
    \begin{itemize}
        \item if $p = q \pm 1$ then both $c_p\delta^q$ and $c_q\delta^p$ partition $P \setminus \{p,q\}$ into $l+1$ singletons, and
        \item otherwise both $c_p\delta^q$ and $c_q\delta^p$ are the error message $\ast$.
    \end{itemize} 
    %Next, we construct an unfillable $(l+2)$-cycle on the $(l+3)$-element set $P= \Z/(l+3)\Z$ for $l\geq 2$. (This proves the cases $l=2$.) 
    %For each $p\in \Z/(l+3)\Z$, define an equivalence relation $c_p$ on $P \setminus \{p\}$ by
    % l=2 だけかく?
    %\[
    %ac_p b \iff (a+1=p=b-1) \lor (b+1=p=a-1) \lor (a=b).
    %\]
    %In other words, $c_p$ partitions $P \setminus \{p\}$ into a single two-element set $\{p-1,p+1\}$ and $l$ singletons.
    %Since there are exactly $l+1$ equivalence classes, we may regard $c_p$ as an element of $\Eq_{=l+1}\bigl(P \setminus \{p\}\bigr)$.
    %These $\Eq_{=l+1}$-structures satisfy the cycle equations because, for distinct $p,q \in \Z/(l+3)\Z$,
    %\begin{itemize}
        %\item if $p = q \pm 1$ then both $c_p\delta^q$ and $c_q\delta^p$ partition $P \setminus \{p,q\}$ into $l+1$ singletons, and
        %\item otherwise both $c_p\delta^q$ and $c_q\delta^p$ are the error message $\ast$.
    %\end{itemize} 
    % \memo{visualise this!} 
    This cycle is unfillable because (under the obvious abuse of notation) $0,2 \in P$ are $c_1$-related but not $c_3$-related.

    Lastly, we construct an unfillable $(2l-1)$-cycle on a $2l$-element set $P$ for $l\geq 3$. Let us fix a decomposition $P = A\sqcup B$ with $|A|=|B|=l$. For $p \in P$, define an equivalence relation $c_p$ on $P \setminus \{p\}$ by
    \[
    xc_py \iff (x\in A \land y\in A \land p\in A) \lor  (x\in B \land y\in B \land p\in B) \lor (x=y). 
    \]
    In other words, $c_p$ partitions $P \setminus \{p\}$ into
    \begin{itemize}
        \item a single $(l-1)$-element set $A \setminus \{p\}$ and $l$ singletons if $p \in A$, and 
        \item a single $(l-1)$-element set $B \setminus \{p\}$ and $l$ singletons if $p \in B$.
    \end{itemize}
    These $\Eq_{=l+1}$-structures satisfy the cycle equations because, for distinct $p,q \in P$,
    \begin{itemize}
        \item if $p,q \in A$ then both $c_p\delta^q$ and $c_q\delta^p$ partition $P \setminus \{p,q\}$ into a single $(l-2)$-element set $A \setminus \{p,q\}$ and $l$ singletons,
        \item similarly in the case $p,q \in B$, and
        \item otherwise (that is, if either $p \in A$ and $q \in B$, or $p \in B$ and $q \in A$) both $c_p\delta^q$ and $c_q\delta^p$ are the error message $\ast$.
    \end{itemize}
    % \memo{visualise this as well!}
    This cycle is unfillable because, if we take distinct $p_1,p_2,p_3 \in A$ and $q \in B$, then $p_1$ and $p_2$ are $c_{p_3}$-related but not $c_q$-related.
\end{proof}

These cycles will be visualised in \Cref{fig:ReductionGraphOfLowerBound} after defining the notion of reduction graph (\Cref{def:reductionGraph}).

\begin{remark}[Comparison with \cite{kennett2011levels}]
    The family of unfillable cycles for $l\geq 3$ in the proof of \Cref{lem:LowerBoundcycle} is inspired by \cite[Example 3.20]{kennett2011levels}.
    More precisely, our construction essentially extracts the (simplicial version of the) EZ-congruence from that example.
\end{remark}

\begin{proposition}[Lower bound]\label{prop:LowerBound}
    For $l \ge 1$, we have
    \[
    a_l \geq
    \begin{cases}
%        0 &(l= -\infty)\\
%        1&(l= 0)\\
        2&(l= 1)\\
        4&(l= 2)\\
        2l-1&(l\geq 3).
    \end{cases}
    \]
    % \[
    % \begin{cases}
    %     a_n \geq0 &(n= -\infty)\\
    %     1&(n= 0)\\
    %     3&(n= 1)\\
    %     4&(n= 2)\\
    %     2n-1&(n\geq 3)
    % \end{cases}
    % \]
    % $a_{-\infty} = 0, a_{0}=1, a_{1}={3}, a_2 = 4, a_{\infty} = \infty$.
\end{proposition}

\section{Computing the Aufhebung}
We gave a lower bound for the Aufhebung in \Cref{prop:LowerBound}.
The aim of this section is to prove that the Aufhebung is bounded above by, and hence equal to, those values.

\subsection{Preparation: Propagative graphs}\label{subsection:PropagativeGraphs}
In this subsection, we introduce a purely graph-theoretic phenomenon (\Cref{prop:GraphCalculation}), which will eventually provide a sharp upper bound for the Aufhebung (\Cref{prop:ReductionGraphIsPropagative}).
Let us start by considering the following 
% graph-theoretic
notion, which may be visualised as \Cref{fig:propagative graph}.
\begin{definition}
    Given an undirected graph $G = \bigl(V, E \subset \Pow_2(V)\bigr)$, we write $\Phi = \Phi_G \colon \Pow(V) \to \Pow(V)$ for the operation given by
    \[
    \Phi S = S \cup \bigl\{v \in V \mid \exists s_0, s_1 \in S\; (s_0 \neq s_1 \wedge \{s_0, v\},  \{s_1, v\}\in E)\bigr\}.
    \]
    We call $G$ \demph{propagative} if, for any $S \subset V$ with $|S|=2$,
    % $|S| \ge 2$ 
    % \inred{[This (as opposed to the original version $|S|=2$) is more convenient when proving \Cref{prop:GraphCalculation}]}
    % \horamemo{I still do not understand why $|S|\geq2$ is more convenient. They are equivalent, and I personally feel $|S|=2$ is a bit simpler.}, 
    there exists $k \ge 1$ such that $\Phi^k S = V$.
\end{definition}

\begin{figure}[ht]
    \centering
    \begin{tikzpicture}[every node/.style={circle, draw, minimum size=0.5cm, inner sep=0}, node distance=2cm]

% Step 1
\begin{scope}[shift={(0, 0)}]
% Nodes
\node[fill=white] (0) at ({cos(360/7 * 0)},{sin(360/7 * 0)}) {};
\node[fill=white] (1) at ({cos(360/7 * 1)},{sin(360/7 * 1)}) {};
\node[fill=white] (2) at ({cos(360/7 * 2)},{sin(360/7 * 2)}) {};
\node[fill=blue!30] (3) at ({cos(360/7 * 3)},{sin(360/7 * 3)}) {};
\node[fill=white] (4) at ({cos(360/7 * 4)},{sin(360/7 * 4)}) {};
\node[fill=blue!30] (5) at ({cos(360/7 * 5)},{sin(360/7 * 5)}) {};
\node[fill=white] (6) at ({cos(360/7 * 6)},{sin(360/7 * 6)}) {};

% Edges
\draw (0) -- (1);
\draw (0) -- (2);
\draw (0) -- (3);
\draw (0) -- (4);
\draw (0) -- (5);
\draw (0) -- (6);
\draw (1) -- (2);
\draw (1) -- (3);
\draw (1) -- (4);
\draw (1) -- (5);
\draw (1) -- (6); 
\draw (3) -- (4);

% Label
\node[draw = none] at (0,-1.8) {$S$};
\end{scope}

% Step 2
\begin{scope}[shift={(6, 0)}]
% Nodes
\node[fill=blue!30] (0) at ({cos(360/7 * 0)},{sin(360/7 * 0)}) {};
\node[fill=blue!30] (1) at ({cos(360/7 * 1)},{sin(360/7 * 1)}) {};
\node[fill=white] (2) at ({cos(360/7 * 2)},{sin(360/7 * 2)}) {};
\node[fill=blue!30] (3) at ({cos(360/7 * 3)},{sin(360/7 * 3)}) {};
\node[fill=white] (4) at ({cos(360/7 * 4)},{sin(360/7 * 4)}) {};
\node[fill=blue!30] (5) at ({cos(360/7 * 5)},{sin(360/7 * 5)}) {};
\node[fill=white] (6) at ({cos(360/7 * 6)},{sin(360/7 * 6)}) {};

% Edges
\draw (0) -- (1);
\draw (0) -- (2);
\draw (0) -- (3);
\draw (0) -- (4);
\draw (0) -- (5);
\draw (0) -- (6);
\draw (1) -- (2);
\draw (1) -- (3);
\draw (1) -- (4);
\draw (1) -- (5);
\draw (1) -- (6); 
\draw (3) -- (4);

% Label
\node[draw = none] at (0,-1.8) {$\Phi S$};
\end{scope}

% Step 3
\begin{scope}[shift={(12, 0)}]
% Nodes
\node[fill=blue!30] (0) at ({cos(360/7 * 0)},{sin(360/7 * 0)}) {};
\node[fill=blue!30] (1) at ({cos(360/7 * 1)},{sin(360/7 * 1)}) {};
\node[fill=blue!30] (2) at ({cos(360/7 * 2)},{sin(360/7 * 2)}) {};
\node[fill=blue!30] (3) at ({cos(360/7 * 3)},{sin(360/7 * 3)}) {};
\node[fill=blue!30] (4) at ({cos(360/7 * 4)},{sin(360/7 * 4)}) {};
\node[fill=blue!30] (5) at ({cos(360/7 * 5)},{sin(360/7 * 5)}) {};
\node[fill=blue!30] (6) at ({cos(360/7 * 6)},{sin(360/7 * 6)}) {};

% Edges
\draw (0) -- (1);
\draw (0) -- (2);
\draw (0) -- (3);
\draw (0) -- (4);
\draw (0) -- (5);
\draw (0) -- (6);
\draw (1) -- (2);
\draw (1) -- (3);
\draw (1) -- (4);
\draw (1) -- (5);
\draw (1) -- (6); 
\draw (3) -- (4);

% Label
\node[draw = none] at (0,-1.8) {$\Phi^2 S$};
\end{scope}

\end{tikzpicture}
    \caption{An example of a propagative graph}
    \label{fig:propagative graph}
\end{figure}
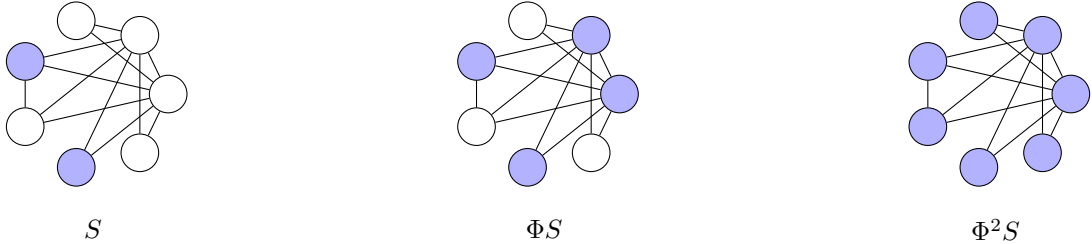

\begin{puzzle}\label{pzl}
% [Puzzle (formal):]
    Fix $n> 2$.
    Find the minimum $m$ such that any undirected graph $G = (V,E)$ with
    \begin{itemize}
        \item $|V| = n$, and
        \item each $v \in V$ has degree at least $m$
    \end{itemize}
    is necessarily propagative.
\end{puzzle}

The following example demonstrates that such $m$ must be strictly greater than $\lfloor\frac{n}{2}\rfloor$.

\begin{example}\label{example}No bipartite graph (with more than $2$ vertices) is propagative because taking one vertex from each part yields a two-element subset $S$ with $\Phi S = S$.
In particular, for any $n>2$, there is a non-propagative $n$-vertex graph such that each vertex has degree at least $\lfloor\frac{n}{2}\rfloor$ (\Cref{fig:BipartiteGraph}).

\begin{figure}[ht]
    \centering
     \begin{tikzpicture}[every node/.style={circle, draw, minimum size=0.5cm, inner sep=0}]

    % 左側の頂点 (集合 U, 4つ)
    \foreach \i in {1,...,4} {
        \node[fill=white] (u\i) at (0, -\i) {};
    }

    % 右側の頂点 (集合 V, 3つ)
    \foreach \j in {1,...,3} {
        \node[fill=white] (v\j) at (4, -\j-0.5) {};
    }

    % 特定の頂点に色をつける
    \node[fill=blue!30] at (0, -1) {};
    \node[fill=blue!30] at (4, -1.5) {};

    % 辺を引く (完全二部グラフ)
    \foreach \i in {1,...,4} {
        \foreach \j in {1,...,3} {
            \draw (u\i) -- (v\j);
        }
    }

    \end{tikzpicture}
    \caption{Bipartite graphs are never propagative.}
    \label{fig:BipartiteGraph}
\end{figure}
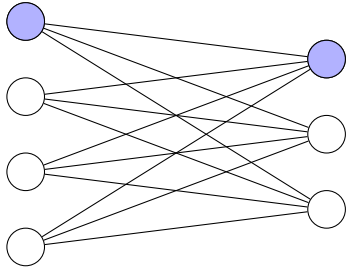

\end{example}

This bound is in fact sharp; that is, the answer to \Cref{pzl} is $m=\lfloor\frac{n}{2}\rfloor +1$.

\begin{proposition}\label{prop:GraphCalculation}
    Let $n> 2$ and let $G = (V,E)$ be an undirected graph with $n$ vertices.
    If every vertex in $G$ has degree greater than $\frac{n}{2}$, then $G$ is propagative.
\end{proposition}
\begin{proof} (This proof is visualised in \Cref{fig:IllustrationForTheProof}.) 
Let $m$ be the minimum degree of the vertices in $G$.
We will prove the contrapositive of the proposition. 
More precisely, we will assume that the graph $G$ is not propagative, and prove $2m\leq n$.

Since $G$ is not propagative, we can find a proper subset $S \subsetneq V$ such that $|S| \ge 2$ and $\Phi S = S$.
Observe that the equation $\Phi S = S$ is equivalent to the condition
\begin{itemize}
    \item [(\dag)] each vertex $v\in V \setminus S$ has at most one neighbour in $S$.
\end{itemize}
We claim that the cardinality $x \coloneqq |S|$ satisfies
\begin{itemize}
    \item[(i)] $2\leq x\leq n-m$, and
    \item[(ii)] $x(x-1)+(n-x)\geq xm$.
\end{itemize}

For the second inequality in (i), pick $v \in V \setminus S$.
Then its degree $d_v$ satisfies $d_v \ge m$ by the minimality of $m$, and it also satisfies 
\[
d_v \leq  |V\setminus (S \cup \{v\})|+1 = n-x
\]
by (\dag).
Combining these inequalities yields $x \le n-m$.
To obtain (ii), consider the sum of degrees of all elements in $S$.
On the one hand, we can deduce from (\dag) that it is bounded above by
\[
|S|\bigl(|S|-1\bigr) + |V\setminus S| = x(x-1) + (n-x).
\]
On the other hand, the minimality of $m$ provides a lower bound $|S|m = xm$. This completes the proof of the two inequalities (i) and (ii).

% \begin{itemize}
%     \item $2\leq x\leq n-m$,
%     \begin{itemize}
%         \item 
%         % We have $x\leq n-m$, 
%         because for any element $v\in G\setminus S$, its degree $d_v$ should satisfy $m\leq d_v \leq  |G\setminus (S \cup \{v\})|+1 = n-x$. 
%     \end{itemize}
%     \item $x(x-1)+(n-x)\geq xm$.
%     \begin{itemize}
%         \item 
%         % $x(x-1)+(v-x)\geq xd$ 
%         because the sum of degrees of all elements in $S$ has an upper bound $|S|(|S|-1) + |G\setminus S| = x(x-1) + (n-x)$ and a lower bound $|S|m = xm$.
%     \end{itemize}
% \end{itemize}  
%Especially, we have $2\leq n-m$.
Since the function $f(t) =t(t-1)+(n-t)- tm$ is downward convex, the existence of a solution $x$ to the inequalities (i) and (ii) implies the satisfaction of (ii) at one of the boundaries of (i).
If it is satisfied at $x=2$ then we have
\[
n = 2(2-1)+(n-2) \ge 2m.
\]
In the other case, rearranging the resulting inequality
\[
(n-m)(n-m-1)+(n-(n-m)) \ge (n-m)m
\]
yields
\[
(n-2m)(n-m-1)\geq 0.
\]
Since (i) implies $n-m-1 > 0$, we conclude that $n \ge 2m$ holds in either case.
% \begin{description}
%     \item[$f(2)\geq 0$] $ \iff 2(2-1)+(n-2)- 2m\geq 0 \iff n\geq 2m$ or
%     \item[$f(n-m)\geq 0$] $(n-m)(n-m-1)+(n-(n-m))- (n-m)m \; \iff(n-2m)(n-m-1)\geq 0 \iff n\geq 2m$.
% \end{description}
% At the last equivalence, we use $n-m\geq 2$.
% , since otherwise, we have $n= m+1$ and $G$ is a complete graph, which contradicts the assumption that $G$ is not propagative.
% If $n-m=1$, then the graph $G$ is a complete graph, then $G$ is obviously propagative. We may assume $n-m>1$.
% Since $2\leq x\leq v-d$ implies $v-d-1 >0$, the above two conditions are equivalent to $v\geq 2d$.
% This contradicts the assumption $v<2d$.
%In both cases, we have $n\geq 2m$.
\end{proof}

\begin{figure}[ht]
    \centering
\begin{tikzpicture}[every node/.style={circle, draw, minimum size=0.5cm, inner sep=0}]

% 頂点の座標
\def\radius{2} % 半径
\foreach \i in {0,...,5} {
    \node[fill={white}, draw] (v\i) at ({\radius*cos(-60+60*\i)}, {\radius*sin(-60+60*\i)}) {};
}

% 左側3頂点を白塗り
\foreach \i in {0,1,2} {
    \node[fill=white] (v\i) at ({\radius*cos(-60+60*\i)}, {\radius*sin(-60+60*\i)}) {};
}

% \node[fill={white}, draw] (v6) at (\radius+ \radius,0) {};
\node[fill=white] (v6) at (\radius+ \radius,0) {};

% 右側3頂点を青塗り
\foreach \i in {3,4,5} {
    \node[fill=blue!30] (v\i) at ({\radius*cos(-60+60*\i)}, {\radius*sin(-60+60*\i)}) {};
}

% 辺を描く
% 左側の完全グラフ (白点)
\draw (v0) -- (v1);
\draw (v1) -- (v2);
\draw (v2) -- (v0);

% 右側の完全グラフ (青点)
\draw (v3) -- (v4);
\draw (v4) -- (v5);
\draw (v5) -- (v3);

% 同じ高さの点を結ぶ
\draw (v0) -- (v5);
\draw (v1) -- (v4);
\draw (v2) -- (v3);

\draw (v0) -- (v6);
\draw (v1) -- (v6);
\draw (v2) -- (v6);

% 左側を囲む (S)
\draw[thick, dashed] (-3,-2.5) rectangle (0,2.5);
\node[draw = none] at (-2,2) {$S$};

% 右側を囲む (G \setminus S)
\draw[thick, dashed] (0,-2.5) rectangle (5,2.5);
\node[draw = none] at (4,2) {$G \setminus S$};

\end{tikzpicture}
    \caption{Illustration for $n=7, m=3$.}
    \label{fig:IllustrationForTheProof}
\end{figure}
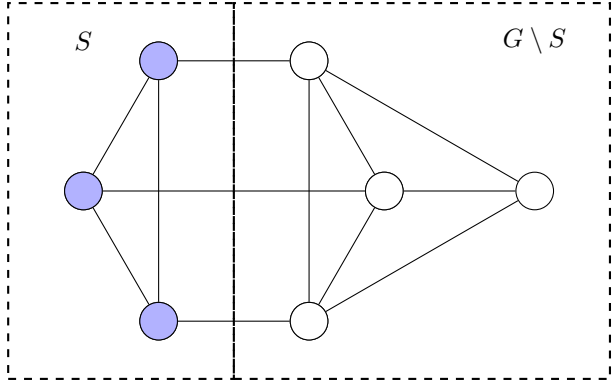

\subsection{Reduction at points}

In later subsections, where we will be dealing with highly degenerate $M$-structures $x$ on a finite set $A$, it will be useful to keep track of which points of $A$ may be discarded without losing essential information about $x$.
This will be done using the following notion of \emph{reduction}, and this subsection is devoted to proving various technical lemmas around it.

\begin{lemma}\label{lem:EquivalentConditionsForTheReduction}
    Let $M$ be a symmetric set, $A$ be a non-empty finite set, $a\in A$ be an element, and $x\in M(A)$ be an $M$-structure.
    Then the following conditions are equivalent.
    \begin{enumerate}
        \item $\mass(x\delta^a) = \mass(x)$.
        \item For any EZ-decomposition $x=y\alpha$ of $x$, the decomposition $x\delta^a = y (\alpha\delta^a)$ of $x\delta^a$ is an EZ-decomposition.
        \item For some EZ-decomposition $x=y\alpha$ of $x$, the decomposition $x\delta^a = y (\alpha\delta^a)$ of $x\delta^a$ is an EZ-decomposition.
        \item The $\sim_x$-class containing the element $a \in A$ is not a singleton.
        % \memo{Usozyanai}
    \end{enumerate}
\end{lemma}
\begin{proof}
    % To prove $(1)\implies (2)$, we show $\lnot (2) \implies \lnot (1)$. Suppose that the equivalence class of $a$ with respect to $\sim_x$ is not a singleton. By \Cref{lem:laxCongruence}, we obtain $\mass(x\delta^a)< \mass(x)$.
    Assume $(1)$, and let $x = y\alpha$ be an EZ-decomposition of $x$.
    Then the decomposition $x\delta^a = y (\alpha \delta^a)$ realizes the mass $\mass(x\delta^a) = \mass(x)$, which implies $(2)$ by \Cref{lem:CriterionOfEZDecomposition}.

    Clearly $(2)$ implies $(3)$.
    
    The condition $(3)$ implies that the composite $\alpha \delta^a$ is surjective, which in turn implies $(4)$.

    Assuming $(4)$, we can take $a \neq a'\in A$ such that $a\sim_x a'$.
    Fix an EZ-decomposition $x=y\alpha$, and write $\rho\colon A \twoheadrightarrow A\setminus \{a\}$ for the retraction of $\delta^a$ that maps $a$ to $a'$.
    Then we have $\alpha(a) = \alpha(a')$, which implies $\alpha\delta^a\rho = \alpha$, which in turn implies 
    $x\delta^a \rho = y\alpha\delta^a\rho = y\alpha = x$.
    Thus by \Cref{lem:pbdoesntIncreaseMass}, we have 
    \[
    \mass(x) \geq \mass(x\delta^a) \geq \mass(x\delta^a \rho) = \mass(x).
    \]
    This completes the proof.
\end{proof}

We will borrow the following terminology from \cite{kennett2011levels}.

\begin{definition}[Reduction]\label{def:reduction}
    For a symmetric set $M$, a non-empty finite set $A$, and an $M$-structure $x\in M(A)$, we say that $a\in A$ \demph{reduces} $x$ if $a$ and $x$ satisfy the equivalent conditions in \Cref{lem:EquivalentConditionsForTheReduction}.
\end{definition}

% Since we will consider filling an $n$-dimensional $k$-cycle with $n<<k$, our strategy is to analyse the abundance of degeneracies...\memo{write}

Intuitively, a point $a\in A$ reduces an $M$-structure $x\in M(A)$ if $x\delta^a$ retains enough information about $x$; \Cref{fig:DecompositionLifting} provides an example of reduction/non-reduction in the case where $M=\Graph$. 
In the rest of this subsection, we will prove lemmas which make this intuition precise.
% This notion of reduction will be repeatedly utilized in this paper.

\begin{lemma}[Congruence lifting]\label{lem:restrictingcongruence}
    Let $M$ be a symmetric set, $A$ be a non-empty finite set, $a \in A$ be an element, and $x\in M(A)$ be an $M$-structure.
    If $a$ reduces $x$, then the EZ-congruence associated to $x\delta^a$ is the restriction of that associated to $x$.
\end{lemma}
\begin{proof}
    This immediately follows from the second (or third) condition in \Cref{lem:EquivalentConditionsForTheReduction}.
    % from \Cref{lem:laxCongruence} and $\mass(x) = \mass(x\delta^a)$.
\end{proof}

\begin{lemma}[Equality lifting: one point]\label{lem:EqualityLiftingOne}
    Let $M$ be a symmetric set, $A$ be a non-empty finite set, $a \in A$ be an element, and $x,y\in M(A)$ be $M$-structures.
    Suppose that
    \begin{itemize}
        \item $x$ and $y$ have the same EZ-congruence,
        \item $a$ reduces either (hence both) of $x$ and $y$, and
        \item $x\delta^a=y\delta^a$.
    \end{itemize}
    Then we have $x=y$.
\end{lemma}
\begin{proof}
    Let us write $\sim$ for the common EZ-congruence of $x$ and $y$.
    Then we can take EZ-decompositions $x=z\pi$ and $y=w \pi$ involving the same quotient map $\pi\colon A \twoheadrightarrow A/{\sim}$.
    The second assumption implies that $\pi\delta^a$ is surjective, whereas the third implies
    \[
    z(\pi \delta^a) =x\delta^a=y\delta^a =w (\pi \delta^a).
    \]
    We have thus obtained two EZ-decompositions of the same $M$-structure, so applying \Cref{prop:UniquenessOfEZdecomposition} yields an automorphism $\sigma$ on $A/{\sim}$ such that $\sigma\pi\delta^a = \pi\delta^a$ and $z = w\sigma$.
    But since $\pi\delta^a$ is surjective, the first of these two equalities implies that $\sigma$ is the identity, which in turn implies $z=w$.
    This completes the proof.
    % \memo{teineini}
\end{proof}

\begin{lemma}[Equality lifting:  two points]\label{lem:EqualityLiftingTwoPoints}
    Let $M$ be a symmetric set, $A$ be a non-empty finite set, $a,b \in A$ be elements, and $x,y\in M(A)$ be $M$-structures.
    Suppose that
    \begin{itemize}
        \item each of $a$ and $b$ reduces both $x$ and $y$,
        \item $x\delta^a = y\delta^a$, and
        \item $x\delta^b=y\delta^b$.
    \end{itemize}
    Then we have $x=y$.
\end{lemma}
\begin{proof}
    By \Cref{lem:EqualityLiftingOne}, it suffices to prove the equality between their EZ-congruences ${\sim_x} = {\sim_y}$.
    Combining the second and third assumptions with \Cref{lem:restrictingcongruence}, we can deduce $c \sim_x d \iff c \sim_y d$ for all $c,d \in A$ except for the pair $\{c,d\} = \{a,b\}$.
    It thus remains to prove $a\sim_x b \iff a\sim_y b$. 

    Suppose $a\sim_x b$. If the $\sim_x$-class containing $a$ and $b$ contains a third element $c$, then we have
    \[
    a \sim_{x\delta^b} c \sim_{x\delta^a} b
    \quad\text{or equivalently}\quad
    a \sim_{y\delta^b}  c \sim_{y\delta^a} b,
    \]
    which implies $a\sim_y b$ by \Cref{lem:restrictingcongruence}.
    % This follows from
    % \[
    % a\sim_x b \iff a \sim_{x\delta^a} \exists c \sim_{x\delta^y} b \iff a \sim_{y\delta^a} \exists c \sim_{y\delta^y} b \iff a\sim_y b
    % \]
    If this $\sim_x$-class is just $\{a,b\}$, then $a$ is isolated with respect to ${\sim_{x\delta^b}}={\sim_{y\delta^b}}$.
    However, $a$ cannot be isolated with respect to ${\sim_{y}}\subset A^2$ because $a$ reduces $y$. This implies $a\sim_y b$ as desired.
\end{proof}

The following lemma is visualised in \Cref{fig:DecompositionLifting} in the case of $M = \Graph$.
\begin{lemma}[Decomposition lifting]\label{lem:DecompositionLifting}
    Let $M$ be a symmetric set, $A$ be a non-empty finite set, $a \in A$ be an element, and $x\in M(A)$ be an $M$-structure.
    Suppose that $a$ reduces $x$, and fix an EZ-decomposition $x\delta^a = y\alpha$ of $x\delta^a$ where $\alpha \colon A\setminus\{a\}\twoheadrightarrow B$.
    Then there exists a unique surjection $\beta\colon A \twoheadrightarrow B$ such that
    \begin{itemize}
        \item $x= y\beta$ is an EZ-decomposition of $x$, and
        \item $\beta$ is an extension of $\alpha$ (that is, $\beta \delta^a =  \alpha$).
    \end{itemize}
\end{lemma}
\begin{proof}
    We will first prove the uniqueness of such $\beta$, assuming that it exists, by applying \Cref{lem:EqualityLiftingOne} to the representable symmetric set $\F(-,B)$.
    Observe that, when $\beta$ is regarded as an $\F(-,B)$-structure on $A$, the associated EZ-congruence $\sim_\beta$ is given by $b \sim_\beta c \iff \beta(b) = \beta(c)$ for $b,c \in A$.
    Hence $\sim_\beta$ coincides with $\sim_x$, and in particular $a$ reduces $\beta$.
    Moreover, pulling back the $M$-structure $\beta$ along $\delta^a$ simply yields the composite $\beta\delta^a$, which is required to be $\alpha$.
    This completes the proof of the uniqueness.

    Now we construct the desired $\beta$.
    Let $x= z\gamma$ be an EZ-decomposition of $x$.
    Since $a$ reduces $x$, we obtain another EZ-decomposition $x\delta^a = z (\gamma \delta^a)$ of $x\delta^a$.
    By \Cref{prop:UniquenessOfEZdecomposition}, there exists a bijection $\sigma$ such that $\sigma \gamma \delta^a = \alpha$ and $z= y\sigma$.
    % \begin{itemize}
    %     \item .
    % \end{itemize}
    It is easy to see that $\beta \coloneqq \sigma \gamma$ satisfies the required conditions.
\end{proof}

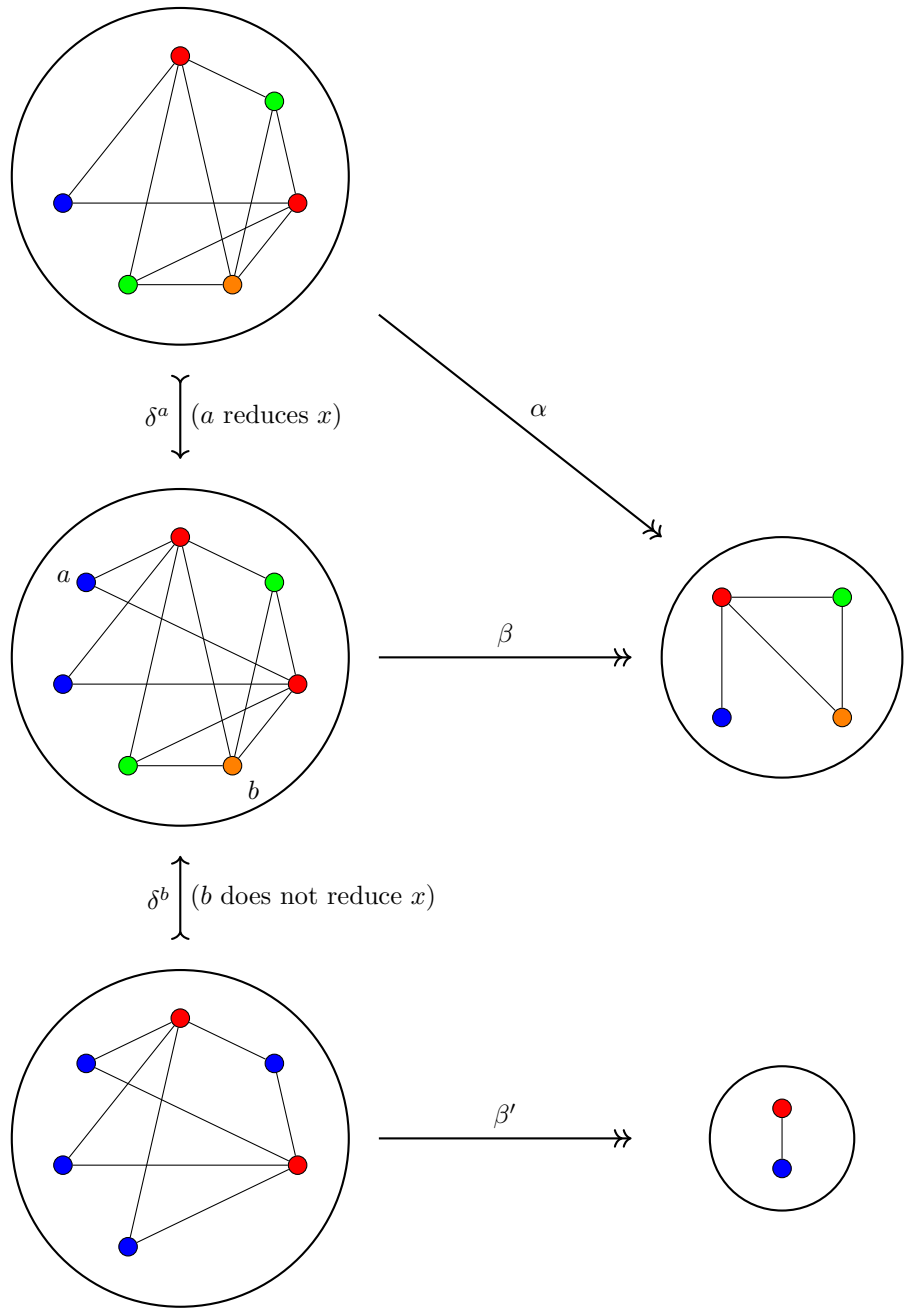
\begin{figure}[ht]
    \centering
    \begin{tikzpicture}[scale=0.8]
    % 半径設定
    \def\r{2} % 各グラフの描画範囲の半径
    \def\RcircleLeft{2.8} % 左のグラフを囲む円の半径
    \def\RcircleRight{2} % 右のグラフを囲む円の半径
    \def\Xshift{10} % 左右のグラフの間隔
    \def\Yshift{8} % 左上の6頂点グラフの高さ
    \def\YshiftDown{-8} % 左下の6頂点グラフの高さ
    \def\XshiftRight{10} % 右下の2頂点グラフの横移動
    \def\RcircleSmall{1.2} % 2頂点グラフを囲む円の半径

    % 頂点の色を設定
    \def\vertexcolor#1#2{%
        \ifnum#2=4 % 新しい6頂点グラフ（頂点4を除外）
            \ifnum#1=3 blue\else%
            \ifnum#1=6 blue\else%
            \ifnum#1=0 red\else%
            \ifnum#1=5 red\else%
            \ifnum#1=1 blue\else%
            \ifnum#1=2 blue\else%
            black\fi\fi\fi\fi\fi\fi%
        \else % その他のグラフ
            \ifnum#1=0 red\else%
            \ifnum#1=5 red\else%
            \ifnum#1=1 blue\else%
            \ifnum#1=2 blue\else%
            \ifnum#1=3 green\else%
            \ifnum#1=6 green\else%
            \ifnum#1=4 orange\else%
            black\fi\fi\fi\fi\fi\fi\fi%
        \fi%
    }

    % 左側の7頂点グラフ
    \begin{scope}
        % 左のグラフを囲む円
        \draw[thick] (0, 0) circle (\RcircleLeft);

        % 頂点を配置
        \foreach \i in {0,...,6} {
            \node[draw, circle, fill=\vertexcolor{\i}{0}, minimum size=7pt, inner sep=0pt] (v\i) at ({90 + \i * 360 / 7}:\r) {};
        }

        % 頂点1の左にラベル"a"を追加
        \node[left] at ($(v1) + (-0.1, 0.1)$) {\(a\)};

        % 頂点4の右下にラベル"b"を追加（7頂点グラフ内）
        \node[below right] at ($(v4) + (0.1, -0.1)$) {\(b\)};

        % 辺を描画
        \draw (v0) -- (v1);
        \draw (v0) -- (v2);
        \draw (v0) -- (v3);
        \draw (v0) -- (v4);
        \draw (v0) -- (v6);
        \draw (v1) -- (v5);
        \draw (v2) -- (v5);
        \draw (v3) -- (v4);
        \draw (v3) -- (v5);
        \draw (v4) -- (v5);
        \draw (v4) -- (v6);
        \draw (v5) -- (v6);
    \end{scope}

    % 左側の6頂点グラフ（上に追記）
    \begin{scope}[shift={(0, \Yshift)}] % 上に移動
        % 6頂点グラフを囲む円
        \draw[thick] (0, 0) circle (\RcircleLeft);

        % 頂点を配置（頂点1を除外）
        \foreach \i in {0,...,6} {
            \ifnum\i=1
                % 頂点1を除外
            \else
                \node[draw, circle, fill=\vertexcolor{\i}{0}, minimum size=7pt, inner sep=0pt] (u\i) at ({90 + \i * 360 / 7}:\r) {};
            \fi
        }

        % 辺を描画（頂点1を除外した6頂点グラフ）
        \foreach \i/\j in {0/2, 0/3, 0/4, 0/6, 2/5, 3/4, 3/5, 4/5, 4/6, 5/6} {
            \draw (u\i) -- (u\j);
        }

        % 矢印を追加（6頂点グラフから7頂点グラフへ）
        \draw[>->, thick] (0, -\RcircleLeft-0.5) -- (0, -\Yshift + \RcircleLeft+0.5) 
            node[midway, left] {\(\delta^a\)}
            node[midway, right] {($a$ reduces $x$)};
    \end{scope}

    % 左側の6頂点グラフ（下に追加）
    \begin{scope}[shift={(0, \YshiftDown)}] % 下に移動
        % 6頂点グラフを囲む円
        \draw[thick] (0, 0) circle (\RcircleLeft);

        % 頂点を配置（頂点4を除外）
        \foreach \i in {0,...,6} {
            \ifnum\i=4
                % 頂点4を除外
            \else
                \node[draw, circle, fill=\vertexcolor{\i}{4}, minimum size=7pt, inner sep=0pt] (d\i) at ({90 + \i * 360 / 7}:\r) {};
            \fi
        }

        % 辺を描画（頂点4を除外した6頂点グラフ）
        \foreach \i/\j in {0/1, 0/2, 0/3, 0/6, 1/5, 2/5, 3/5, 5/6} {
            \draw (d\i) -- (d\j);
        }

        % 矢印を追加（6頂点グラフから7頂点グラフへ）
        \draw[>->, thick] (0, \RcircleLeft+0.5) -- (0, -\YshiftDown - \RcircleLeft-0.5) 
            node[midway, left] {\(\delta^b\)}
            node[midway, right] {($b$ does not reduce $x$)};
    \end{scope}

    % 右下の2頂点グラフ
    \begin{scope}[shift={(\Xshift, \YshiftDown)}]
        % 2頂点グラフを囲む円
        \draw[thick] (0, 0) circle (\RcircleSmall);

        % 2頂点を配置
        \node[draw, circle, fill=red, minimum size=7pt, inner sep=0pt] (p0) at (0, 0.5) {};
        \node[draw, circle, fill=blue, minimum size=7pt, inner sep=0pt] (p1) at (0, -0.5) {};

        % 辺を描画（2頂点間の辺）
        \draw (p0) -- (p1);
    \end{scope}

    % 矢印を追加（左下の6頂点グラフから右下の2頂点グラフへ）
    \draw[->>, thick] (\RcircleLeft + 0.5, \YshiftDown) -- (\Xshift - \RcircleRight - 0.5, \YshiftDown) node[midway, above] {\(\beta'\)};

    % 右側のグラフ
    \begin{scope}[shift={(\Xshift,0)}] % 右に移動
        % 右のグラフを囲む円
        \draw[thick] (0, 0) circle (\RcircleRight);

        % 頂点を正方形に配置
        \node[draw, circle, fill=red, minimum size=7pt, inner sep=0pt] (w0) at (-1, 1) {};
        \node[draw, circle, fill=blue, minimum size=7pt, inner sep=0pt] (w1) at (-1, -1) {};
        \node[draw, circle, fill=orange, minimum size=7pt, inner sep=0pt] (w2) at (1, -1) {};
        \node[draw, circle, fill=green, minimum size=7pt, inner sep=0pt] (w3) at (1, 1) {};

        % 辺を描画
        \draw (w0) -- (w1);
        \draw (w0) -- (w2);
        \draw (w0) -- (w3);
        \draw (w2) -- (w3);
    \end{scope}

    % 矢印を追加（左から右へ）
    \draw[->>, thick] (\RcircleLeft + 0.5, 0) -- (\Xshift - \RcircleRight - 0.5, 0) node[midway, above] {$\beta$};

    % 矢印を追加（左上から右下へ）
    \draw[->>, thick] (0.5+\RcircleLeft, \Yshift - \RcircleLeft + 0.5) -- (\Xshift - \RcircleRight , \RcircleRight ) node[midway, above right] {$\alpha$};

\end{tikzpicture}
    \caption{Decomposition lifting}
    \label{fig:DecompositionLifting}
\end{figure}

\subsection{Reduction graph}\label{subs:ReductionGraph}
% \memo{reconstuct}

Armed with all the preparation, we are now ready to start proving that any cycle that comprises sufficiently degenerate $M$-structures admits a unique degenerate filler (\Cref{filling}).
This will provide the desired upper bound for the Aufhebung.

\begin{notation}\label{notation:TheCycleNotation}
Throughout Subsections \ref{subs:ReductionGraph} to \ref{section_exceptional},
% \inred{the rest of this paper?}\horamemo{Almost yes, but I want to exclude the main theorem}, 
we fix
    \begin{itemize}
        \item a symmetric set $M$,
        \item a positive integer $k > 0$,
        \item a $(k+1)$-element set $P$, and
        \item a $k$-cycle $\{c_p\}_{p \in P}$ (satisfying the cycle equations $c_p \delta^q = c_q \delta^p$),
    \end{itemize}
and write
\begin{itemize}     
    \item $n\coloneqq \max\{\mass(c_p)\mid p\in P\}-1$,
    \item $d\coloneqq k-n$,
    \item $\uP\coloneqq \{p\in P \mid \mass{(c_p)}=n+1\}$, and
    \item  $\dP\coloneqq \{p\in P \mid \mass{(c_p)}\leq n\}$.
\end{itemize}
% \horamemo{I will write the bound of $d$ in each propositions.}
% Moreover, we will assume $d\geq 2$. 
% \horamemo{I will write the bound of $d$ in each propositions.}
\end{notation}
Note that we have $P = \uP\sqcup \dP$ and $\uP \neq \emptyset$.

The majority of what follows is devoted to proving that, if the following inequalities hold, then the $k$-cycle $\{c_p\}_{p \in P}$ admits a unique degenerate filler:
%Given that we know the lower bound (\Cref{prop:LowerBound}), we will construct a filler under the following three inequalities, which will be referred to as \CrefIneq:
% \horamemo{I have erased the first inequality $n\geq 1$. We don't need this. Without this, the logical structure of \Cref{thm:MainTheorem} becomes a little bit clearer.}
% \begin{subequations} \label{eq:all_ineq}
% \begin{align}
%     % n &\geq 1 \label{eq:ineq1} \\
%     d &\geq 3\text{ (i.e., $k>n+2$)}\\
%     % \label{eq:ineq2} \\
%     k &> 2n - 1 
%     % \label{eq:ineq3}
% \end{align}
% \end{subequations}
\begin{equation} \label{eq:all_ineq}
    \begin{aligned}
    d &\geq 3 \;(\iff k>n+2) \\ 
    % \text{ (i.e., $k>n+2$)}\\
    k &> 2n - 1 .
\end{aligned}\tag{$\bigstar$}
\end{equation}
This covers all necessary cases except for $(n,k) = (1,3)$, which will be treated separately in \Cref{section_exceptional}.
%When we use this inequality, we will explicitly assume it.
We emphasise that these inequalities are NOT assumed by default in what follows; when they are assumed, we will explicitly state so.

An example of $\{c_p\}_{p \in P}$ satisfying {\CrefIneq} in the case where $M = \Graph$ and $(n,k,d)=(3,6,3)$ is visualised in \Cref{fig:CycleFilling}.
%The reduction graphs of the cycles in \Cref{lem:LowerBoundcycle}, for which $P=\uP$ holds, are visualised in \Cref{fig:ReductionGraphOfLowerBound}.

% \begin{itemize}
%     \item $n\geq 1$
%     \item $d\geq 3$ 
%     \item $k>2n-1$.
% \end{itemize}

% \horamemo{I want to write that "Under the three inequalities $n\geq 1$, $d\geq 3$, and $k>2n-1$, ... }

\begin{definition}\label{def:reductionGraph}
    By the \demph{reduction graph} $G$ of the $k$-cycle $\{c_p\}_{p \in P}$, we mean the directed graph such that
    \begin{itemize}
        \item its vertex set is $P$, and
        \item there is a (unique) edge $p\to q$ if and only if $p$ reduces $c_q$.
    \end{itemize}
    We write $\uG$ for the subgraph of $G$ consisting of the vertices in $\uP$ and all edges between them.
\end{definition}

\begin{figure}
    \centering
    \begin{tikzpicture}
    % 半径設定
    \def\r{2}
    \def\arrowoffset{0.13}

    % 頂点を配置
    \foreach \i in {0,...,6} {
        \ifnum\i=4
            \node[draw, circle, fill=white, minimum size=5pt, inner sep=0pt] (v\i) at ({90 + \i * 360 / 7}:\r) {};
        \else
            \node[draw, circle, fill=black, minimum size=5pt, inner sep=0pt] (v\i) at ({90 + \i * 360 / 7}:\r) {};
        \fi
    }

    % 矢印の描画
    \newcommand{\shortarrow}[2]{
        \draw[<->, thick] ($(#1) !\arrowoffset! (#2)$) -- ($(#2) !\arrowoffset! (#1)$);
    }

    \newcommand{\directarrow}[2]{
        \draw[->, thick] ($(#1) !\arrowoffset! (#2)$) -- ($(#2) !\arrowoffset! (#1)$);
    }

    % 矢印の接続
    \shortarrow{v0}{v1};
    \shortarrow{v5}{v1};
    \shortarrow{v0}{v2};
    \shortarrow{v5}{v2};
    \shortarrow{v0}{v3};
    \shortarrow{v5}{v3};
    \shortarrow{v0}{v6};
    \shortarrow{v5}{v6};
    \shortarrow{v1}{v3};
    \shortarrow{v1}{v6};
    \shortarrow{v2}{v3};
    \shortarrow{v2}{v6};

    \directarrow{v0}{v4};
    \directarrow{v5}{v4};
    \directarrow{v1}{v4};
    \directarrow{v2}{v4};
    \directarrow{v3}{v4};
    \directarrow{v6}{v4};
\end{tikzpicture}
    \caption{Reduction graph of the cycle \Cref{fig:CycleFilling} with black $\uP$ and white $\dP$}
    \label{fig:ReductionGraph}
\end{figure}
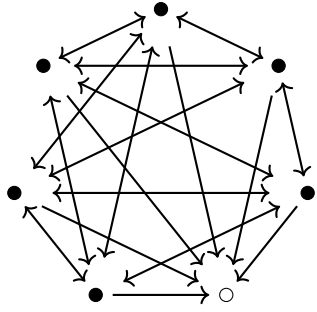

% For a vertex $p\in P$, we write $N(p)$ for the set $\{p\} \cup \{q\in P \mid q \to p\}$.

% \begin{proposition}\label{prop:basicPropertiesOfReductionGraph}
%     The reduction graph $G$ has the following properties:
%     \begin{description}
%         \item[$\uG$ is undirected] for any edge $p\to q$ with $q\in\uP$, we have $p\in \uP$ and $p \leftarrow q$.
%         \item[Degree] For each vertex $p$, its out degree $d_p$ is more than or equal to $d$.
%         \begin{itemize}
%             \item If $d_p=d \geq 2$, $p\in \uP$.
%             \item If $d_p=d\geq 3$, $N(p)\subset \uG$ is a clique.
%         \end{itemize}
%     \end{description}
% \end{proposition}

\begin{proposition}\label{prop:basicPropertiesOfReductionGraph}
    The reduction graph $G$ has the following properties.
    \begin{enumerate}
        \item There is no edge from a vertex in $\dP$ to one in $\uP$.
        \item In the subgraph $\uG$, there is an edge $p \to q$ if and only if there is an edge $p \leftarrow q$.
    \end{enumerate}
    In the case $d \geq 2$, it further enjoys the following properties.
    \begin{enumerate}[resume]
        \item The indegree $d_p$ of each vertex $p \in P$ is at least $d$.
        \item The subgraph $\uG$ contains at least $d+1$ vertices.
    \end{enumerate}
    % \begin{description}
    %     \item[One directed] There are no edges from $\dP$ to $\uP$.
    %     \item[$\uG$ is undirected]  In $\uG$, $p\to q$ iff $p\leftarrow q$.
    %     \item[Degree] If $d\geq 2$, for each vertex $p$, its indegree $d_p$ is greater than or equal to $d$.
    %     \invmemo{
        
    %         \item Furthermore, if $d_p=d$, then $p\in \uP$.
    %         \begin{itemize}and $N(p)\subset \uG$ is a $(d+1)$-clique.
    %         % \item If $d_p=d\geq 3$, $N(p)\subset \uG$ is a clique.
    %     \end{itemize}
    %     }
    %     \item[$\uG$ is large] If $d\geq 2$, $|\uG|\geq d+1$.
    % \end{description}
\end{proposition}
\begin{proof}
    To see (1) and (2), consider an edge $p\to q$ with $q\in \uP$.
    Since $p$ reduces $c_q$, we have
    \[
    \mass(c_p) \geq \mass(c_p \delta^q) = \mass(c_q \delta^p) = \mass(c_q) = n+1 \ge \mass(c_p).
    \]
    In particular, we can deduce $\mass(c_p) = n+1$ and $\mass(c_p\delta^q) = \mass(c_p)$.

    Now suppose $d \ge 2$.
    For each vertex $p \in P$, its indegree $d_p$ is equal to 
    \[
    d_p = k-\text{(the number of isolated points with respect to } {\sim_{c_p}}\text{)}.
    \]
    Recall that $\sim_{c_p}$ is an equivalence relation on the $k$-element set $P \setminus \{p\}$, and the number of $\sim_{c_p}$-classes is precisely $\mass(c_p) \le n+1$.
    Since we are assuming $k = n+d \ge n+2$, there must be at least one $\sim_{c_p}$-class which is not a singleton.
    Hence the number of isolated points in the above equation is at most $n$, which establishes (3).

    For (4), fix $q \in \uP$, whose existence follows from the definition of $n$.
    Then there are at least $d$ edges of the form $p \to q$ by (3), and each such $p$ belongs to $\uP$ by (1).
    This completes the proof.
    %     \begin{description}
    %     \item[One directed] Consider an edge $p\to q$ with $q\in \uP$. Since $p$ reduces $c_q$, we have $\mass(c_p) \geq \mass(c_p \delta^q) = \mass(c_q \delta^p) = \mass(c_q) = n+1$, which implies that $p\in \uP$.
    %     \item[$\uG$ is undirected]  Consider an edge $p\to q$ in $\uG$. We have $\mass(c_p\delta^q) = \mass(c_q \delta^p) = \mass (c_q)=n+1$, which implies $q \to p$.
    %     \item[Degree] \invmemo{Here we use the assumption $d\geq 3$} For each vertex $p$, its indegree $d_p$ is equal to 
    %     \[
    %     d_p = k-\text{(the number of isolated points with respect to } {\sim_{c_p}}\text{)}.
    %     \]
    %     Since $\mass(c_p)\leq n+1$
    %     % $c_p$ is $n$-skeletal
    %     and $d\geq 2$, the number of isolated points is at most $n$. This proves $d_p \geq k-n=d$.
    %     \item[$\uG$ is large] By the definition of $n$, $\uG$ is not empty. For a vertex $p\in \uG$, its indegree is greater than or equal to $d$, and all the sources of these edges are in $\uG$.
    % \end{description}
\end{proof}

\begin{figure}
    \centering
    \begin{tikzpicture}[scale = 0.9]
    \def\r{1.7}
    \def\arrowoffset{0.13}
    \def\xshift{3.5}
    \def\yshift{4}

    \newcommand{\drawPolygon}[3]{
        \foreach \i in {0,...,#1} {
            \node[draw, circle, fill=#3, draw=#3, minimum size=5pt, inner sep=0pt] (v\i) at ({(\i - 0.25* (#1+3)) * (360 / (#1+1))}:\r) {};
        }
        \foreach \i in {0,...,#1} {
            \pgfmathtruncatemacro{\j}{mod(\i+1,#1+1)}
            \draw[<->, thick, color=#3] ($(v\i) !\arrowoffset! (v\j)$) -- ($(v\j) !\arrowoffset! (v\i)$);
        }
    }

    \newcommand{\drawSplitPolygon}[2]{
        \pgfmathtruncatemacro{\half}{(#1+1)/2 -1}
        \pgfmathtruncatemacro{\halfu}{(#1+1)/2}
        \foreach \i in {0,...,#1} {
            \node[draw, circle, fill=#2, draw=#2, minimum size=5pt, inner sep=0pt] (v\i) at ({(\i + 0.25* (#1+3)) * (360 / (#1+1))}:\r) {};
        }
        \foreach \i in {0,...,\half} {
            \foreach \j in {0,...,\half} {
                \pgfmathtruncatemacro{\jindex}{\j} \pgfmathtruncatemacro{\iindex}{\i}
                \ifnum\iindex<\jindex \relax
                    \draw[<->, thick, color=#2] ($(v\iindex) !\arrowoffset! (v\jindex)$) -- ($(v\jindex) !\arrowoffset! (v\iindex)$);
                \fi
            }
        }
        \foreach \i in {\halfu,...,#1} {
            \foreach \j in {\halfu,...,#1} {
                \pgfmathtruncatemacro{\jindex}{\j} \pgfmathtruncatemacro{\iindex}{\i}
                \ifnum\iindex<\jindex \relax
                    \draw[<->, thick, color=#2] ($(v\iindex) !\arrowoffset! (v\jindex)$) -- ($(v\jindex) !\arrowoffset! (v\iindex)$);
                \fi
            }
        }
    }

    \begin{scope}[shift={(2*\xshift, 2* \yshift)}]
        \foreach \i in {0,...,2} {
            \node[draw, circle, fill=black, draw=black, minimum size=5pt, inner sep=0pt] (v\i) at ({(\i + 0.25* (4)) * (360 / (3))}:\r) {};
        }
    \end{scope}

    \begin{scope}[shift={(2*\xshift, \yshift)}]
        \drawPolygon{4}{black}{black}
    \end{scope}

    \begin{scope}[shift={(2*\xshift, 0)}]
        \drawSplitPolygon{5}{black}
    \end{scope}
    \begin{scope}[shift={(2*\xshift, -\yshift)}]
        \drawSplitPolygon{7}{black}
    \end{scope}
    \begin{scope}[shift={(2*\xshift, -2*\yshift)}]
        \drawSplitPolygon{9}{black}
    \end{scope}

    \node at (\xshift, 2*\yshift) {$l=1$};
    \node at (\xshift, \yshift) {$l=2$};
    \node at (\xshift, 0) {$l=3$};
    \node at (\xshift, -\yshift) {$l=4$};
    \node at (\xshift, -2*\yshift) {$l=5$};
    \node at (1.5 * \xshift, {(-2.7)*\yshift}) {$\vdots$};

    \end{tikzpicture}
    \caption{The reduction graphs of the cycles in \Cref{lem:LowerBoundcycle}, which are not propagative.}
    \label{fig:ReductionGraphOfLowerBound}
\end{figure}
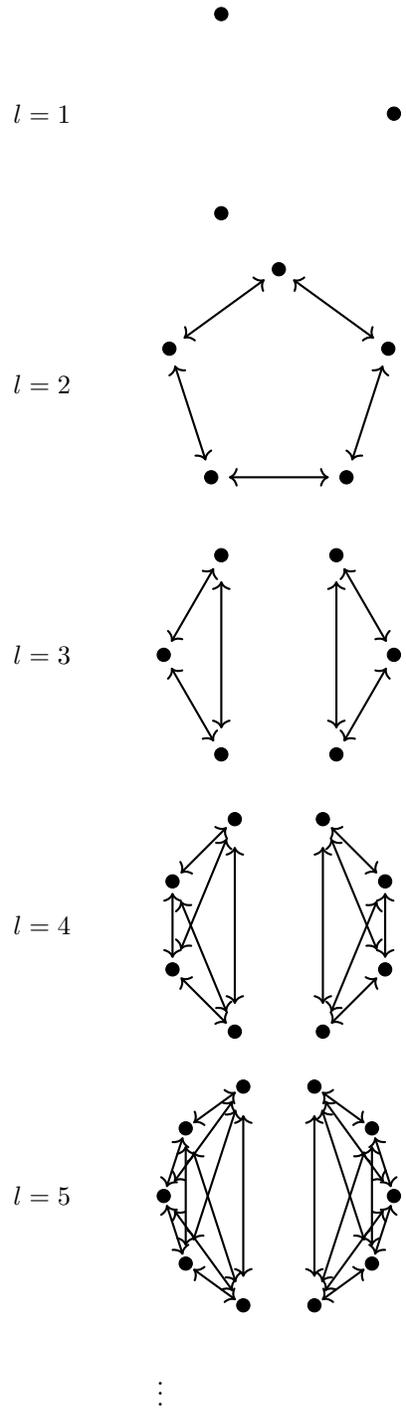

Notice that we can regard the graph $\uG$, which is a priori directed, as an undirected graph thanks to \Cref{prop:basicPropertiesOfReductionGraph}(2).

\begin{remark}[The reduction graphs of the unfillable cycles in \Cref{lem:LowerBoundcycle}]\label{rmk:ReductionGraphOfTheUnfillableCycles}
It turns out that, assuming \CrefIneq, the undirected graph $\uG$ must be propagative (\Cref{prop:ReductionGraphIsPropagative}), and this property will play a central role when constructing a filler of $\{c_p\}_{p \in P}$ (\Cref{prop:KnowTwotoKnow}).
Hence, in order to find an unfillable $k$-cycle $\{c_p\}_{p \in P}$, it is natural to look for one with a non-propagative reduction graph.
Indeed, the cycles appearing in \Cref{lem:LowerBoundcycle} (for which $P=\uP$ holds) have non-propagative reduction graphs, as visualised in \Cref{fig:ReductionGraphOfLowerBound}.
% (Those not achieving the sharp lower bound are de-emphasised.)
\end{remark}

\invmemo{this was subsumed by claim a:\\
Note that, if $p \in P$ satisfies $d_p= d$, then the EZ-congruence $\sim_{c_p}$ consists of a $d$-element equivalence class $K$ and $n$ isolated points. This implies $\mass(c_p) = n+1$, or equivalently $p \in \uP$. 
}
\invmemo{Furthermore, for any $q\in K$, $K\setminus\{q\}$ is EZ-congruent in $c_q$ (since $q\leftrightarrow p$). This \invmemo{and the assumption $d\geq 3$} implies that any other $q' \in K\setminus\{q\}$ reduces $c_q$.}
% distinct $q, q'\in K$, we can take 

\subsection{Candidate filler and its compatibility with \texorpdfstring{$\uP$}{uP}}\label{subs:CandidateFillerAndItsCompatibilityWithUp}
% \memo{We will 
% % not 
% use the uniqueness proven in this subsection.}
%First we will prove that, when $k$ is sufficiently large compared to $n$, there exists $f \in M(P)$ that satisfies $f\delta^p = c_p$ for $p \in \uP$. Our strategy is applying \Cref{prop:GraphCalculation} to the reduction graph $\uG$.

% So our strategy is:
% \begin{enumerate}
%     \item Filling two points as an initial state (\Cref{lem:fillingEdge}),
    
%     \item  Verifying the propagation step (\Cref{prop:KnowTwotoKnow}), and
%     \item Proving that the reduction graph $\uG$ is propagative. (\Cref{prop:ReductionGraphIsPropagative})
% \end{enumerate}

We wish to show that, %under the assumption that $k$ is sufficiently larger than $n$,
assuming \CrefIneq, we may construct a filler $f$ of the cycle $\{c_p\}_{p \in P}$.
In this subsection, we construct a candidate for such $f$, and prove that it satisfies $f\delta^p = c_p$ for at least all $p \in \uP$.
Our strategy is to prove that
\begin{enumerate}
    \item we can find two points $p,q \in \uG$ and $f \in M(P)$ such that $f\delta^p = c_p$ and $f\delta^q = c_q$ (\Cref{lem:fillingEdge}),
    \item the property ``$f\delta^r = c_r$'' propagates along the edges of $\uG$ (\Cref{prop:KnowTwotoKnow}), and
    \item $\uG$ is propagative (\Cref{prop:ReductionGraphIsPropagative}).
\end{enumerate}
Moreover, such $f$ turns out to be unique if we impose $\mass(f) = n+1$.

\begin{lemma}[Seed of propagation]\label{lem:fillingEdge}
    Let $p\leftrightarrow q$ be an edge in $\uG$.
    Then there exists a unique $f \in M(P)$ such that $f \delta^p = c_p$, $f\delta^q = c_q$, and $\mass(f)=n+1$.
\end{lemma}
\begin{proof}
    Observe that, since $p,q \in \uG$ implies $\mass(c_p) = \mass(c_q) = n+1$, any solution $f$ to these three equations is necessarily reduced by both $p$ and $q$.
    It follows by \Cref{lem:EqualityLiftingTwoPoints} that such $f$ is unique if it exists.
    
    Let $c_p \delta^q = c_q \delta^p = x\alpha$ be an EZ-decomposition where $\alpha \colon P\setminus\{p,q\} \twoheadrightarrow S$. 
    Since $p$ reduces $c_q$ and $q$ reduces $c_p$, \Cref{lem:DecompositionLifting} provides surjections $\beta_p\colon P\setminus\{p\}\twoheadrightarrow S$ and $\beta_q \colon P\setminus\{q\}\twoheadrightarrow S$ such that
    \begin{itemize}
        \item $c_p = x\beta_p$,
        \item $\beta_p \delta^q = \alpha$,
        \item $c_q = x\beta_q$, and
        \item $\beta_q \delta^p = \alpha$.
    \end{itemize}
    We define $\beta \colon P \twoheadrightarrow S$ to be the unique common extension of $\beta_p$ and $\beta_q$. Then $f \coloneqq x\beta$ satisfies $f \delta^p = c_p$, $f \delta^q = c_q$, and $\mass(f)=n+1$.
\end{proof}

In the following proposition, note that we are allowing the case $r \in \dP$.
This extra generality will be utilised in the next subsection.
% So for the rest of this subsection, we will prove the followings:

\begin{proposition}[Propagation step]\label{prop:KnowTwotoKnow}
    Let $p,q,r \in P$ be distinct elements and $f \in M(P)$ with $\mass(f)=n+1$.
    Suppose that
    \begin{itemize}
        \item $p,q \in \uP$,
        \item both $p$ and $q$ reduce $c_r$,
        \item $f \delta^p = c_p$, and
        \item $f\delta^q = c_q$.
    \end{itemize}
    Then we have $f\delta^r = c_r$.
\end{proposition}
\begin{proof}
    % We may assume $r\neq p,q$ since the proposition becomes trivial otherwise.
    We wish to apply \Cref{lem:EqualityLiftingTwoPoints} to $x=f\delta^r$ and $y=c_r$.
    Since we already know $f\delta^r  \delta^p = f\delta^p  \delta^r = c_p \delta^r = c_r \delta^p$ and similarly $f\delta^r  \delta^q = c_r \delta^q$, it remains to prove that $p$ and $q$ reduce $f \delta^r$.
    Note that, since we already know $\mass(f\delta^r  \delta^p) =\mass (c_r \delta^p)= \mass (c_r)$ and similarly $\mass(f\delta^r\delta^q) = \mass(c_r)$, $p$ reduces $f\delta^r$ if and only if $\mass(f\delta^r)= \mass(c_r)$ if and only if $q$ reduces $f\delta^r$.
    % what we must prove is $\mass(f\delta^r) \leq \mass (c_r)$.

    % For suppose $\mass(f\delta^r) > \mass (c_r)$. Then neither $p$ nor $q$ reduces $f\delta^r$, which means that both $p$ and $q$ are isolated with respect to ${\sim_{f\delta^r}}$. However, they are not isolated with respect to $\sim_f$, as they reduce $f$ ($\mass(f)= n+1 =\mass(c_p) = \mass(f \delta^p)$). This situation and \Cref{lem:laxCongruence} imply $p \sim_f r \sim_f q$, which contradicts $p \not \sim_{f\delta^r} q$.
    Suppose for contradiction that neither $p$ nor $q$ reduces $f\delta^r$.
    Then $p$ is isolated with respect to ${\sim_{f\delta^r}}$.
    However, it is not isolated with respect to $\sim_f$ as it reduces $f$ (because $\mass(f)= n+1 =\mass(c_p) = \mass(f \delta^p)$).
    Therefore we have $p \sim_f r$ by \Cref{lem:laxCongruence}, and similarly $q \sim_f r$, which implies $p \sim_f q$.
    This in turn implies $p \sim_{f\delta^r} q$ again by \Cref{lem:laxCongruence}, which contradicts the second sentence of this paragraph.
\end{proof}

\begin{lemma}\label{lem:TwoCases}
    % \memo{Old version: If $k>2n-1$ and $d\geq 3$,
    % % $n\geq3$, 
    % then }
    % \horamemo{new version: Under the three inequalities $n\geq 1$, $d\geq 3$, and $k>2n-1$,}
    Assuming \CrefIneq,
    at least one of the following conditions holds.
    \begin{enumerate}
        \item For every vertex $p\in P$, we have $d_p>d$.
        \item The set $\dP$ contains at least $n-1$ elements.
    \end{enumerate}
\end{lemma}
\begin{proof}
    We may assume $n\geq 2$ because the condition (2) trivially holds in the case $n \le 1$.
    We will further assume
    \begin{itemize}
        \item[ ($\lnot$1)] there exists $p \in P$ with $d_p = d$, and
        \item[ ($\lnot$2)] the set $\dP$ contains less than $n-1$ elements,
    \end{itemize}
    and derive a contradiction.
    For fix $p \in P$ with $d_p = d$, and define
    \[
    Q\coloneqq \{q\in P\mid q\to p\}.
    % = \{q\in \uP\mid q \leftrightarrow p\}
    \]
    % where the second equality follows from \Cref{prop:basicPropertiesOfReductionGraph}. 
    % \memo{The "second equality" follows from the fact that $p\in P$. So this should be put after the first claim.}
    We will write $R\coloneqq P \setminus (Q\cup \{p\})$.
    Note that this set $R$ contains exactly $(k+1)-(d+1) = n$ elements.
    \begin{claim}\label{sim_c_p}
        The EZ-congruence $\sim_{c_p}$ partitions $P \setminus \{p\}$ as
        \[
        P\setminus \{p\} = Q \sqcup\{r_1\} \sqcup \dots \sqcup \{r_n\}
        \]
        where $r_i \in R$, so in particular $p \in \uP$.
        Consequently, for any $q \in Q$, the EZ-congruence $\sim_{c_{q}\delta^{p}}$ partitions $P \setminus \{p, q\}$ as
        \[
        P\setminus \{p,q\} = Q \setminus \{q\} \sqcup\{r_1\} \sqcup \dots \sqcup \{r_n\}.
        \]
    \end{claim}
    \begin{proof}[Proof of Claim]
        We know that none of the $n$ elements $r_i \in R$ reduces $c_p$, so \Cref{lem:EquivalentConditionsForTheReduction} implies that each $\{r_i\}$ is a singleton $\sim_{c_p}$-class.
        Since there can only be $\mass(c_p) \le n+1$ many $\sim_{c_p}$-classes, it follows that the whole $Q$ forms a single $\sim_{c_p}$-class.
        The last assertion follows from the cycle equation $c_q\delta^p = c_p\delta^q$ 
        % \Cref{lem:EquivalentConditionsForTheReduction}, 
        and \Cref{lem:restrictingcongruence}.
    \end{proof}
    As a consequence of $p\in \uP$, we obtain another description of the set $Q$:
    \[
    Q = \{q\in \uP\mid q \leftrightarrow p\}
    \]
    due to \Cref{prop:basicPropertiesOfReductionGraph}.
    \begin{claim}\label{ReductionAmongQ}
        Let $q,q' \in Q$ be distinct elements.
        Then $q'$ reduces $c_q$.
    \end{claim}
    \begin{proof}[Proof of Claim]
        Since $p$ reduces $c_q$, the EZ-congruence $\sim_{c_q\delta^p}$ of \Cref{sim_c_p} is the restriction of $\sim_{c_q}$ onto $P \setminus \{p,q\}$.
        In particular, the $\sim_{c_q}$-class containing $q'$ must subsume $Q \setminus \{q\}$.
        It therefore follows from $|Q \setminus \{q\}| =d-1 \ge 3-1 = 2$
        %and \Cref{lem:EquivalentConditionsForTheReduction}
        that $q'$ reduces $c_q$.
    \end{proof}
    \begin{claim}\label{ReducingNone}
        The set $R$ contains at least $n-1$ elements that reduce no $c_q$ for $q \in Q$.
    \end{claim}
    \begin{proof}[Proof of Claim]
        First, let us fix $q \in Q$ and consider which $r_i \in R$ does not reduce $c_q$.
        Since $p$ reduces $c_{q}$, the EZ-congruence $\sim_{c_q\delta^p}$ of \Cref{sim_c_p} is the restriction of $\sim_{c_{q}}$ onto $P \setminus \{p, q\}$.
        It follows that $\{r_i\}$ is a singleton $\sim_{c_q}$-class for all $1 \le i \le n$ except for possibly one $i$ satisfying $r_i \sim_{c_q} p$.

        It remains to prove that if we have $r_i \sim_{c_q} p$ and $r_{i'} \sim_{c_{q'}} p$ for distinct $q,q' \in Q$ then $i=i'$.
        Observe that, since $q$ reduces $c_{q'}$ and $q'$ reduces $c_q$ by \Cref{ReductionAmongQ}, the EZ-congruences $\sim_{c_q}$ and $\sim_{c_{q'}}$ restrict to the same equivalence relation $\sim$ on $P \setminus \{q,q'\}$.
        So $r_i \sim_{c_q} p$ and $r_{i'} \sim_{c_{q'}} p$ would imply $r_i \sim p \sim r_{i'}$, which in turn implies $r_i \sim_{c_q} r_{i'}$, and consequently $r_i \sim_{c_q\delta^p} r_{i'}$.
        We can thus conclude $i=i'$ by \Cref{sim_c_p}.
    \end{proof}

    \begin{claim}\label{r}
        There exists $r \in \uP \cap R$ such that neither $p$ nor any $q \in Q$ reduces $c_r$.
    \end{claim}
    \begin{proof}[Proof of Claim]
        Since we are assuming that $\dP = P \setminus \uP$ contains less than $n-1$ elements, \Cref{ReducingNone} and the Pigeonhole Principle imply that there exists $r \in \uP \cap R$ that reduces no $c_q$ for $q \in Q$.
        We also know that $r$ does not reduce $c_p$ because $r \in R$.
        By \Cref{prop:basicPropertiesOfReductionGraph}(2), this $r$ has the desired property.
    \end{proof}
    
    We can finally derive the desired contradiction using $r$ of \Cref{r} as follows.
    By \Cref{lem:EquivalentConditionsForTheReduction}, each element of $Q \cup \{p\}$ gives rise to a singleton $\sim_{c_r}$-class.
    In addition, since we have
    \[
    |P \setminus \{r\}| = k = n+d > n+1 = \mass(c_r),
    \]
    there must be at least one non-singleton $\sim_{c_r}$-class.
    Hence by counting the number of $\sim_{c_r}$-classes, we obtain the inequality
    \[
    n+1 = \mass(c_r) \geq |Q \cup \{p\}|+1 = d+2 = k-n+2,
    \]
    which is equivalent to $k \le 2n-1$.
    This contradicts \CrefIneq.
\end{proof}

\begin{proposition}\label{prop:ReductionGraphIsPropagative}
    % \memo{old version: If $k>2n-1$ and $n\geq 3$, then} \horamemo{Under the three inequalities $n\geq 1$, $d\geq 3$, and $k>2n-1$,} 
     Assuming \CrefIneq,
    the undirected graph $\uG$ is propagative.
\end{proposition}
\begin{proof}
% Let $v$ be the number of vertices $v \coloneqq \uP$, and 
    Let $m$ be the minimum among the degrees of all vertices in the undirected graph $\uG$.
    By \Cref{prop:GraphCalculation}, it suffices to prove $|\uP|<2m$.
    We will divide our proof into two cases according to \Cref{lem:TwoCases}.
    \begin{enumerate}
        \item Suppose that we have $d_p>d$ for all $p \in P$.
        Then we have $m \ge d+1$.
        Combining this inequality with the assumption $k>2n-1$, or equivalently $k-2n+1 > 0$, we obtain
        \[
        |\uP| \leq |P| = k+1 <   2k-2n+2 = 2(d+1) \leq 2m.
        \]
        \item Suppose that $\dP = P \setminus \uP$ contains at least $n-1$ elements.
        Then, since we are assuming $d \ge 3$, we have
        \[
        |\uP| = |P| - |\dP| \leq (k+1)-(n-1) = d+2 < 2d \leq2 m
        \]
        where the last inequality follows from \Cref{prop:basicPropertiesOfReductionGraph}(3).
    \end{enumerate}
    % \horamemo{When $n=1$, we have $|\uP|=k+1=d+2<2d\leq 2m$.}
    % \horamemo{When $n=2$, 
    % there are two cases. If $d_p >d$, we have $|\uP|\leq k+1\leq k+1+(k-5) =2d<2(d+1)$. If $|P \setminus \uP| \geq n-1$ we have $|\uP|\leq k<2k-4 =2d$.}
    This completes the proof.
\end{proof}

Combining the results in this subsection, we obtain the following.
\begin{proposition}\label{FillinguP}
    % % If $k>2n-1$ and $n\geq 3$ \horamemo{or ($n=1,2$ and $d\geq 3$)},
    % \horamemo{Under the three inequalities $n\geq 1$, $d\geq 3$, and $k>2n-1$,}
     Assuming \CrefIneq,
    there exists a unique\footnote{If we omit $\mass(f)=n+1$, we lose the uniqueness. For example, consider $\Eq_{=j+1}$ for a large $j$. } $f\in M(P)$ such that $\mass(f) = n+1$ and $f\delta^p = c_p$ for all $p\in \uP$.
\end{proposition}
\begin{proof}
    Take an edge $p\leftrightarrow q$ in $\uG$, whose existence is guaranteed By \Cref{prop:basicPropertiesOfReductionGraph}.
    By \Cref{lem:fillingEdge}, there exists $f\in M(P)$ with $\mass(f) = n+1$ such that $f\delta^p = c_p$ and $f\delta^q = c_q$.
    We define $S\subset \uP$ by $S \coloneqq \{r\in \uP\mid f\delta^r = c_r\}$.
    By \Cref{prop:KnowTwotoKnow} and \Cref{prop:ReductionGraphIsPropagative}, we have $S=\uP$.
The uniqueness also follows from \Cref{lem:fillingEdge}.
\end{proof}

\subsection{Compatibility with \texorpdfstring{$\dP$}{dP}}\label{subs:CompatibilityWithDp}
In the previous subsection, we saw that the property ``$f\delta^r=c_r$'' (for the judiciously chosen $f$) propagates to the whole of $\uP$.
It thus remains to prove that this property propagates to $\dP$ too.

\begin{proposition}\label{prop:ReducingDownPoints}
    % If $k>2n-1$ and $d\geq 2$,
    % and $n\geq3$,
    % \horamemo{Under the three inequalities $n\geq 1$, $d\geq 3$, and $k>2n-1$,}\horamemo{For this particular proposition, it suffices to assume $k>2n-1$ and $d\geq 2$.}
     Assume \CrefIneq and let $p\in \dP$.
    Then there exist distinct points $p_0, p_1 \in \uP$ that reduce $c_p$.
    % with $p_0 \to p \leftarrow p_1$.
\end{proposition}
\begin{proof}
    Since $d_p= d$ would imply $p \in \uP$ by \Cref{sim_c_p}, we have $d_p \geq d+1$.
    We also have $|\uP|\geq d+1$ by \Cref{prop:basicPropertiesOfReductionGraph}, so we obtain
    \[
    |\uP|+ d_p \ge 2d+2 = 2(k-n)+2 = k+1+\bigl(k-(2n-1)\bigr) > k+1 = |P \setminus \{p\}|+1.
    \]
    Thus the subsets $\uP$ and $\{q \in P \mid q \to p\}$ of $P \setminus \{p\}$ share at least two elements $p_0$ and $p_1$.
\end{proof}

\begin{theorem}\label{filling}
    % % If $k>2n-1$ and $n\geq 3$ \horamemo{or ($n=1,2$ and $d\geq 3$)}, then
    % \horamemo{Under the three inequalities $n\geq 1$, $d\geq 3$, and $k>2n-1$,}
     Assuming \CrefIneq,
    there exists a unique degenerate $f\in M(P)$ such that $f\delta^p = c_p$ for all $p\in P$.
\end{theorem}
\begin{proof}
    Let $f \in M(P)$ be the $M$-structure of \Cref{FillinguP}.
    We already know that $f\delta^p = c_p$ holds for all $p \in \uP$.
    This property then propagates to all $p \in P$ by \Cref{prop:KnowTwotoKnow} and \Cref{prop:ReducingDownPoints}.
    Thus $f$ is a filler of the cycle $\{c_p\}_{p\in P}$.

    It remains to prove that any degenerate filler $f'$ of $\{c_p\}_{p\in P}$ must coincide with $f$.
    Note that, because of the uniqueness part of \Cref{FillinguP}, it suffices to show that such $f'$ necessarily has mass $n+1$.
    This is indeed the case because $f'$ being degenerate implies 
    \[
    \mass(f') = \max\bigl\{\mass(f'\delta^p) \mid p \in P\bigr\} = \max\bigl\{\mass(c_p) \mid p \in P\bigr\} = n+1.
    \]
    This completes the proof.
\end{proof}

% \horamemo{I will reconstruct these parts.}
% \begin{lemma}[Uniqueness]
%     If a degenerate cell $f\in M(A)$ fills the cycle $\{c_p\}_{p\in P}$, then $f$ satisfies at least one of the following conditions.
%     \begin{itemize}
%         \item $f$ is non-degenerate
%     \end{itemize}
% \end{lemma}

\subsection{The case \texorpdfstring{$n=1$ and $k=3$}{n is one and k is three}}\label{section_exceptional}

\begin{theorem}\label{filling_exception}
    Assuming $n=1$ and $k=3$,
    there exists a unique degenerate $f\in M(P)$ such that $f\delta^p = c_p$ for all $p\in P$.
\end{theorem}
\begin{proof}
    In this case, $\{c_p\}_{p\in P}$ is a $3$-cycle (so $|P| = 4$) such that $\max\{\mass(c_p)\mid p\in P\} = 2$.
    Using the fact $\uP \neq \emptyset$ and \Cref{prop:basicPropertiesOfReductionGraph}, we can pick an edge $p \leftrightarrow q$ in $\uG$.
    By \Cref{lem:fillingEdge}, there exists a unique $f\in M(P)$ such that $f\delta^p =c_p$, $f\delta^q = c_q$, and $\mass(f)=2$.
    We will prove that this $f$ is a filler of the given cycle.
    Note that $p,q \in \uP$ implies
    \[
    \mass(f\delta^p) = \mass(c_p) = 2 = \mass(f)
    \]
    and similarly $\mass(f\delta^q) = \mass(f)$, so both $p$ and $q$ reduce $f$.

    Since $\mass(f) = 2$, there are exactly two $\sim_f$-classes $A,B\subset P$.
    Without loss of generality, we may assume $|A| \ge |B|$.
    Then we have either $|A| = 3$ and $|B| = 1$, or $|A| = |B| = 2$.
    We will treat these two cases separately. 

    Suppose $|A| = 3$ and $|B| = 1$.
    We know that both $p$ and $q$ reduce $f$, or equivalently, neither $p$ nor $q$ forms a singleton $\sim_f$-class, which implies $p,q \in A$.
    We first prove that $f\delta^r = c_r$ holds for the last element $r \in A$.
    Since we have $q \sim_{f\delta^p} r$ by \Cref{lem:laxCongruence}, or equivalently $q \sim_{c_p} r$, this element $r$ reduces $c_p$.
    By \Cref{prop:basicPropertiesOfReductionGraph}(1) and (2), it follows that $p$ reduces $c_r$, and similarly $q$ reduces $c_r$.
    So we can apply \Cref{prop:KnowTwotoKnow} to obtain the desired equation $f\delta^r = c_r$.
    It remains to prove $f\delta^s = c_s$ for the unique element $s \in B$.
    Since the indegree of $s$ in $G$ is at least $d=2$ by \Cref{prop:basicPropertiesOfReductionGraph}(3), this follows similarly from \Cref{prop:KnowTwotoKnow}.
    
    Now we treat the case $|A| = |B| = 2$.
    Without loss of generality, we may assume $p \in A$.
    Then $\sim_{c_p}~=~\sim_{f\delta^p}$ partitions $P \setminus \{p\}$ into $A \setminus \{p\}$ and $B$.
    Since $q$ reduces $c_p$, it follows that we must have $q \in B$.
    Let $r,s \in P$ be the remaining elements so that $A = \{p,r\}$ and $B = \{q,s\}$.
    We will only prove that $f\delta^s = c_s$ holds; the equation $f\delta^r = c_r$ can be proved similarly.

    We will apply \Cref{lem:EqualityLiftingOne} to $A = P \setminus \{s\}$, $a = p$, $x=f\delta^s$, and $y = c_s$.
    We can see from the above description of $\sim_{c_p}$ that $s$ reduces $c_p$, which implies $\mass(c_s) = 2$ and moreover $p$ reduces $c_s$ by \Cref{prop:basicPropertiesOfReductionGraph}(1) and (2) respectively.
    We also have
    \[
    (f\delta^s)\delta^p = (f\delta^p)\delta^s = c_p\delta^s = c_s\delta^p.
    \]
    Therefore it remains to prove that $f\delta^s$ and $c_s$ have the same EZ-congruence.

    Since $s$ reduces $c_p = f\delta^p$, we have
    \[
    \mass(f\delta^s) \ge \mass(f\delta^s\delta^p) = \mass(f\delta^p\delta^s) = \mass(f\delta^p) = \mass(f).
    \]
    Thus $s$ reduces $f$, and it follows by \Cref{lem:restrictingcongruence} that $\sim_{f\delta^s}$ is the restriction of $\sim_f$ onto $P \setminus \{s\}$.
    To prove that this coincides with $\sim_{c_s}$, it suffices to exhibit that $\{q\}$ is a singleton $\sim_{c_s}$-class, or equivalently that $q$ does not reduce $c_s$.
    The latter is indeed the case because the above description of $\sim_f$ implies that
    \[
    c_s\delta^q = c_q\delta^s = f\delta^q\delta^s
    \]
    cannot have mass $2$.

    It remains to prove the uniqueness part.
    This can be done similarly to the proof of \Cref{filling}, by first observing that the mass of any degenerate filler must be $2$, and then combining this observation with the uniqueness part of \Cref{lem:fillingEdge}.
    This completes the proof.
\end{proof}

\subsection{Conclusion}
\begin{theorem}[Main theorem] \label{thm:MainTheorem}
    The Aufhebung of the level labelled by $l$ is given by
    \[
    a_l =
    \begin{cases}
        0 &(l= -\infty)\\
        1&(l= 0)\\
        2&(l= 1)\\
        4&(l= 2)\\
        2l-1&(l\geq 3).
    \end{cases}
    \]
    % $a_{-\infty} = 0, a_{0}=1, a_{1}={3}, a_2 = 4, a_{\infty} = \infty$.
\end{theorem}
\begin{proof}
    The first two values $a_{-\infty}, a_0$ are already established in \Cref{prop:AufhebungMinusInfty} and \Cref{prop:AufhebungZero}.
    % \inred{???}. 
    So fix $l \ge 1$ and let $M$ be an $l$-skeletal symmetric set.
    We will use \Cref{prop:CharacterizationOfCoskeletality} to prove that $M$ is $a_l$-coskeletal for the value of $a_l$ as stated.
    
    Fix
    \[
    k >
    \begin{cases}
        2&(l= 1)\\
        4&(l= 2)\\
        2l-1&(l\geq 3),
    \end{cases}
    \]
    and take an arbitrary $k$-cycle $\{c_p\}_{p\in P}$ in $M$.
    Defining $n$ as in \Cref{notation:TheCycleNotation}, either we are in the case $(n,k)=(1,3)$, or we obtain {\CrefIneq} since the $l$-skeletality of $M$ implies $n\leq l$. 
    Thus the $k$-cycle $\{c_p\}_{p\in P}$ admits a unique degenerate filler by \Cref{filling} and \Cref{filling_exception}.
    In fact, removing the degeneracy assumption does not spoil the uniqueness because the $l$-skeletality forces every filler to be degenerate. 
    
    Combining this observation with the lower bound given in \Cref{prop:LowerBound}, we obtain the Aufhebung relation as stated.
\end{proof}
% \begin{proof}
%     The first two values $a_{-\infty}, a_0$ are already established in \Cref{prop:AufhebungMinusInfty} and \Cref{prop:AufhebungZero}.
%     % \inred{???}. 
%     So fix $n \ge 1$ and let $M$ be an $n$-skeletal symmetric set.
%     Fix
%     \[
%     k >
%     \begin{cases}
%         3&(n= 1)\\
%         4&(n= 2)\\
%         2n-1&(n\geq 3).
%     \end{cases}
%     \]
%     Then {\CrefIneq} 
%     % the three inequalities $n\geq 1$, $k>2n-1$ and $d\geq 3$
%     hold, so every $k$-cycle in $M$ admits a unique degenerate filler by \Cref{filling}.
%     In fact, every $k$-cycle simply admits a unique filler, because the $n$-skeletality forces every filler to be degenerate. Thus, by \Cref{prop:CharacterizationOfCoskeletality}, $M$ is $k$-coskeletal.
%     Combining this observation with the lower bound given in \Cref{prop:LowerBound}, we obtain the Aufhebung relation as stated.
% \end{proof}
\bibliographystyle{alpha}
\bibliography{CommonBiblio}

\end{document}